\documentclass[a4paper]{article}

\usepackage[english,greek]{babel}
\usepackage[utf8]{inputenc}
\usepackage[pdftex,colorlinks]{hyperref}

\usepackage{arxiv}
\usepackage[toc]{appendix} 
\usepackage[utf8]{inputenc} 
\usepackage[T1]{fontenc}  
\usepackage{url}      
\usepackage{booktabs}    
\usepackage{amsfonts}    
\usepackage{nicefrac}    
\usepackage{microtype}   
\usepackage{lipsum}
\usepackage{graphicx}
\graphicspath{ {./images/} }

\usepackage{tikz-cd}

\usepackage{titlesec}

\usepackage{ulem} 

\usepackage{mathptmx}
\usepackage{anyfontsize}

\titleformat*{\section}{\normalsize}
\titleformat*{\subsection}{\small}

\titleformat*{\subsubsection}{\footnotesize}

\titleformat
{\chapter} 
[display] 
{\bfseries\Large\itshape} 
{Story No. \ \thechapter} 
{0.5ex} 
{
    \rule{\textwidth}{1pt}
    \vspace{1ex}
    \centering
} 
[
\vspace{-0.5ex}%
\rule{\textwidth}{0.3pt}
] 

\newcommand\blfootnote[1]{%
  \begingroup
  \renewcommand\thefootnote{}\footnote{#1}%
  \addtocounter{footnote}{-1}%
  \endgroup
}

\usepackage{amsthm}
\usepackage{latexsym} \RequirePackage{amsthm}
\usepackage{amsmath}

\usepackage{amssymb} 
\RequirePackage{makeidx}
\RequirePackage{graphicx}
\usepackage{tgadventor}
\usepackage{txfonts}

\numberwithin{equation}{section}
\DeclareMathOperator{\normdot}{\| \cdot \|}

\DeclareMathOperator{\C}{\mathbb{C}}
\DeclareMathOperator{\N}{\mathbb{N}}
\DeclareMathOperator{\Hh}{ \mathcal{H}}
\DeclareMathOperator{\Kk}{ \mathcal{K}}

\DeclareMathOperator{\Bh}{\mathcal{B}(\mathcal{H})} 
\DeclareMathOperator{\Bk}{\mathcal{B}(\mathcal{K})}
\DeclareMathOperator{\Cs}{\boldsymbol{C}^{*}} 
\DeclareMathOperator{\Aa}{\mathcal{A}} 

\DeclareMathOperator{\Ss}{\mathcal{S}}
\DeclareMathOperator{\Tt}{\mathcal{T}}
\DeclareMathOperator{\Rr}{\mathcal{R}}
\DeclareMathOperator{\Dd}{\mathcal{D}}
\DeclareMathOperator{\Cc}{\mathcal{C}}

\DeclareMathOperator{\B}{\mathcal{B}}

\newcommand{\spann}{\mathrm{span}}

\DeclareMathOperator{\Jj}{\mathcal{J}}
\newcommand{\Mn}{M_{n}}

\newcommand{\rank}{\mathop{\operator@font rank}}

\newcommand{\e}{{\varepsilon}}
\newcommand{\vertiii}[1]{{\left\vert\kern-0.25ex\left\vert\kern-0.25ex\left\vert #1 
    \right\vert\kern-0.25ex\right\vert\kern-0.25ex\right\vert}}
\makeatother

\newcommand{\bb}[1]{\mathbb{#1}}
\newcommand{\cl}[1]{\mathcal{#1}}
\newcommand{\sca}[1]{\langle#1\rangle} %

\newtheorem{theorem}{Theorem}[section]

\newtheorem{proposition}[theorem]{Proposition}
\newtheorem{corollary}[theorem]{Corollary}

\newtheorem{definition}[theorem]{Definition}

\newtheorem{example}[theorem]{Example}

\theoremstyle{remark}

\newtheorem{remark}[theorem]{Remark}

\newtheorem{remarks}[theorem]{Remarks}

\newcommand{\norm}[1]{\left\lVert#1\right\rVert}

\usepackage{hyperref}

\title{ON COPRODUCTS OF OPERATOR $\Aa$-SYSTEMS}

\author{Alexandros Chatzinikolaou}

\begin{document}

\selectlanguage{english}

\maketitle

\begin{abstract}
    {\footnotesize 
   Given a unital $\boldsymbol{C}^{*}$-algebra $\mathcal{A}$, we prove the existence of the coproduct of two faithful operator $\mathcal{A}$-systems. We show that we can either consider it as a subsystem of an amalgamated free product of $\boldsymbol{C}^{*}$-algebras, or as a quotient by an operator system kernel. We introduce a universal $\boldsymbol{C}^{*}$-algebra for operator $\mathcal{A}$-systems and prove that in the case of the coproduct of two operator $\mathcal{A}$-systems, it is isomorphic to the  amalgamated over $\mathcal{A}$, free product of their respective universal $\boldsymbol{C}^{*}$-algebras. Also, under the assumptions of hyperrigidity for operator systems, we can identify the $\boldsymbol{C}^{*}$-envelope of the coproduct with the amalgamated free product of the $\boldsymbol{C}^{*}$-envelopes. We consider graph operator systems as examples of operator $\mathcal{A}$-systems and prove that there exist graph operator systems whose coproduct is not a graph operator system, it is however a dual operator $\mathcal{A}$-system. More generally,  the coproduct of dual operator $\mathcal{A}$-systems is always a dual operator $\mathcal{A}$-system. We show that the coproducts behave well with respect to inductive limits of operator systems.} 
\end{abstract}


\hfill
\section{INTRODUCTION}
The coproduct is a categorical notion that exists for objects in many categories and possesses a universal property. The coproduct of two objects, is a construction that gives rise to a new object to which the latter objects admit morphisms. In the case of $\Cs$-algebras, free products are of great importance for free probability theory \cite{Voiculescu1992FreeRV} and random matrix theory.   Coproducts have also been related to the study of group $ \Cs$-algebras through the connection of free group $ \Cs$-algebras with free products of group $\Cs$-algebras.   Coproducts of operator systems (resp. operator spaces) amalgamated over another operator system (resp. operator space) were introduced in \cite{10.2307/24715886}. As a special case, there is also  the coproduct $ \Ss\oplus_{1}\Tt $ of two operator systems $ \Ss$ and $\Tt$, amalgamated over the unit, which connects to the Connes Embedding Conjecture via its relation to Kirchberg’s conjecture in A. Kavruk's remarkable result. In \cite{Kavruk2014}, Kavruk proved that Kirchberg’s conjecture is equivalent to certain nuclearity properties of a 5-dimensional operator system. The operator system in question, denoted by $ \Ss_{2}$, is in fact an operator system coproduct, i.e. $ \Ss_{2} = \Ss_{1}\oplus_{1} \Ss_{1}$.

The purpose of this paper, is to begin a systematic study of operator $\Aa$-system coproducts  and investigate their connections with other related categories. To this end, we also highlight some of their possible applications.  
Operator $\Aa$-systems are an extension of operator systems which are bimodules over a unital $\Cs$-algebra, say $\Aa$ and  their operator system structure is also compatible with $\Aa $ (operator systems can be seen as operator $\Aa$-systems where $\Aa = \C$). Naturally, their morphisms are  unital completely positive maps that are also $\Aa$-bimodule maps.  In this paper, we begin by proving the existence of coproducts in this category. To do this, we restrict to a smaller class of operator $\Aa$-systems that contain $\Aa$ in some sense. We show that they can be realised as operator subsystems of the free product of $\Cs$-algebras, amalgamated over $ \Aa$ (Theorem \ref{existenceofcop}). The coproduct $ \Ss\oplus_{1}\Tt $ of the operator systems $ \Ss$ and $\Tt$, can be constructed also as an operator system quotient \cite{Fritz2014OperatorSS}, \cite{Kavruk2014} by a completely order  proximinal kernel. We give an analogous representation to the coproduct (Theorem \ref{theorimakernel}) $\Ss\oplus_{\Aa}\Tt$ of operator $ \Aa$-systems $\Ss$ and $\Tt$, as a quotient operator system $ \Ss \oplus \Tt/ \Jj $. An example of operator $\Aa$-systems, are graph operator systems, which have proven to be central in quantum information theory \cite{Paulsen2013QUANTUMCN}. Graph operator systems are operator $\Aa$-systems, with $\Aa= \Dd_{n}$, the $\Cs$-algebra of diagonal matrices in some $ \Mn(\C)$.  So graph operator systems are operator $ \Dd_{n}$-systems. Nevertheless, not all operator $ \Dd_{n}$-systems are graph operator systems. In Proposition \ref{toantiparad} we prove that there exist two $\Dd_{2}$-graph operator systems, whose $\Dd_{2}$-coproduct  is never completely order isomorphic to a graph operator system and in fact, it can never be represented faithfully on a finite dimensional Hilbert space. We do this via  a dimension-type argument, after establishing first a theorem about coproducts and their $\Cs$-envelopes. 

The coproduct $ \Ss \oplus_{1} \Tt$ amalgamated over the unit, for operator systems $ \Ss$ and $\Tt$, can be realised either as a subsystem of the free product of their universal $\Cs$-algebras $ C^{*}_{u}(\Ss) *_{1}C^{*}_{u}(\Tt) $, or as  a subsystem of the free product of their $\Cs$-envelopes $ C^{*}_{e}(\Ss) *_{1}C^{*}_{e}(\Tt) $. Furtheremore,  if  they contain enough unitaries, then $ C^{*}_{e}(\Ss \oplus_{1} \Tt) \cong C^{*}_{e}(\Ss) *_{1}C^{*}_{e}(\Tt)$ \cite{Kavruk2014}. Motivated by the above,  we introduce in Section \ref{sectioncovers} the universal $\Cs$-algebra $ C^{*}_{u,\Aa}(\Ss)$ of an operator $\Aa$-system. With this in hand, we can identify the coproduct of two operator $\Aa$-systems $ \Ss$ and $ \Tt$ as a subsystem of $C^{*}_{u,\Aa}(\Ss) *_{\Aa} C^{*}_{u,\Aa}(\Tt) $, and prove that $C^{*}_{u,\Aa}(\Ss\oplus_{\Aa} \Tt) \cong C^{*}_{u,\Aa}(\Ss) *_{\Aa} C^{*}_{u,\Aa}(\Tt) $. The $\Cs$-envelope is a $\Cs$-cover that respects the module actions, it is thus a suitable $ \Cs$-cover for operator $\Aa$-systems. We prove that if the operator $\Aa$-systems $ \Ss$ and $ \Tt$ are hyperrigid \cite{Arveson2008TheNC} then $ C^{*}_{e}(\Ss \oplus_{\Aa} \Tt) \cong C^{*}_{e}(\Ss) *_{\Aa}C^{*}_{e}(\Tt)$, therefore strengthening the operator system case. 

Dual operator systems, are operator systems that are duals as operator spaces \cite{timoney_2007}. Dual  operator $\Aa$-systems are dual operator systems that also carry an "$\Aa$-compatible" operator system structure. They were introduced in \cite{10.1093/imrn/rnz364}. We prove that the coproducts of  dual operator $\Aa$-systems are also dual operator $\Aa$-systems (Theorem \ref{coproductsofduals}). 

Inductive limits play an important role in $\Cs$-algebra theory and in quantum physics. Their study began by J. Glimm \cite{Glimm1960OnAC} and J. Dixmier \cite{DIXMIER1967182} and they have of course been central in G. Elliott's classification programme \cite{ELLIOTT197629}. Inductive limits of operator systems were introduced by I. Todorov and L. Mawhinney in \cite{Mawhinney2017InductiveLI} where they studied the interactions of inductive limits with operator systems structures as well as with operator system tensor products. Inductive limits of operator $\Aa$-systems have also been developed. In Section \ref{inductivelimits} we consider inductive limits of operator system coproducts, continuing the study began in \cite{Mawhinney2017InductiveLI}. 

\hfill

{\large Acknowledgments.}\\
I would like to express my gratitude to Prof. Aristides Katavolos, for his constant guidance and encouragement. I would also like to thank Profs. Mihalis Anoussis and Ivan Todorov, for their helpful discussions and suggestions. Finally, I would like to thank the anonymous referee for the thorough proofreading and useful comments. This work was partially supported by the Special Account for Research Grants of National and Kapodistrian University of Athens.

\section{PRELIMINARIES}
In this section we establish the terminology and state the definitions that shall be used
throughout the paper.

A \textit{*-vector space} $V$ is a complex vector space, together with an \textit{involution}, i.e. a \textit{conjugate linear} bijection $*: V \rightarrow V $ that is its own inverse. We let  $ V_h= \{ v \in V: v^{*}=v\}$ be the real vector space of 
\textit{hermitian (or selfadjoint)} elements of $V$. 

If $W$ is a real vector space, a \textit{cone} in $W$, is a non-empty subset $ C \subseteq W$ such that:
\begin{enumerate}
    \item if $ w \in C$ and $ \lambda \in [0,+\infty)$, then $ \lambda w \in C$
    \item if $ w, v \in C$, then $ w + v \in C$.
\end{enumerate}
We say that a cone $ C$ is a \textit{proper cone} in $W$ if  $ C\cap -C = \{0\}$.

An \textit{ordered *-vector space} is a pair $ (V,V^+)$, where $V$ is a *-vector space and $ V^+$ is a proper cone in $ V_h$, called the \textit{positive} elements of $V$. The cone  $V^+$ induces a partial order $\geq$ in $V_h $ by $ v \geq w \iff v -w \in V^+$. 

If $(V,V^{+})$ is an ordered *-vector space, an element $ e \in V_{h}$ will be called an \textit{order unit} of $V$, if for every $ v \in V_h$, there exists a real number $ r>0 $ such that $ re \geq v$.
An order unit  $ e \in V_{h}$ is said to be \textit{Archimedean} if whenever $ v \in V_{h}$ satisfies $ r e + v \geq 0$ for every real number $ r>0$, we have that $ v \geq 0$. In this case we call the triple $ (V,V^{+},e)$ an \textit{Archimedean ordered *-vector space}, or an \textit{AOU} space.

For a *-vector space $V$, we denote by $ M_{n,m}(V)$  the set of all n by m matrices with entries from V and set $M_{n,n}(V)= \Mn(V)$. The entry-wise addition and scalar multiplication turn
$M_{n,m}(V)$ into a complex vector space. When $n=m$, we equip $\Mn(V)$ with the involution  $ [v_{i,j}]^{*}=[v_{j,i}^{*}] $ and turn $ \Mn(V)$ into a *-vector space. Denote by $ \Mn(V)_{h}$ the set of hermitian elements of $ \Mn(V)$.
We also set $ M_{n,m}(\C)= M_{n,m}$, $\Mn(\C)= \Mn$ and denote by $ \{ E_{i,j} \}_{i,j=1}^{n}$ the canonical matrix unit system of $ \Mn$. 

\begin{definition}
Let $V$ be a *-vector space. We say that $ \{C_{n}\}_{n\in \N}$ is a matrix ordering on $V$ if: 
\begin{enumerate}
    \item $C_{n}$ is a cone in $\Mn(V)_{h}$ for each $ n\in \N$,
    \item $C_{n}\cap -C_{n}= \{0\}$, for each $ n \in \N$,
    \item $X C_{n} X^{*} \subseteq C_{m}$ for each $ X \in M_{m,n}$ and $ n ,m \in \N$.
\end{enumerate}
In this case we call $ (V,\{C_{n}\}_{n\in \N} )$ a matrix ordered *-vector space. We refer to condition 3. as the compatibility of the family $ \{C_{n}\}_{n\in \N}$.
\end{definition}

If V is a *-vector space and $ \{C_{n}\}_{n\in \N}$ is a matrix ordering on V, then conditions 1. and 2. above, imply that for every $ n\in \N$, the *-vector space $\Mn(V)$ is an ordered *-vector space.

\begin{definition}
Let $ (V,\{C_{n}\}_{n\in \N} )$ be a matrix ordered *-vector space. Let $e \in V_{h}$ and define
\begin{align*}
   e_{n}:= \begin{pmatrix}
        e &  & \\
          & \ddots &\\
          &        &  e
    \end{pmatrix}.
\end{align*}
We say that:
\begin{enumerate}
    \item $e$ is a matrix order unit for $V$, if $e_{n}$ is an order unit for $(\Mn(V), C_{n}) $, for each $n\in \N$, 
    \item $e$ is an Archimedean matrix order unit for $V$, if $e_{n}$ is an Archimedean order unit for $(\Mn(V), C_{n}) $ for each $n\in \N$.
\end{enumerate}
\end{definition}
We may also use the notation $ e_{n}=e \otimes I_{n}$, for the matrix order unit.

\begin{definition}\cite{CHOI1977156} \label{absoperatorsyst}
An (abstract) operator system, is a triple $ (V, \{C_{n}\}_{n\in \N}, e)$, where $(V, \{C_{n}\}_{n\in \N})$ is a matrix ordered *-vector space and $e \in V_{h}$ is an Archimedean matrix order unit.
\end{definition}

Let $ (V, V^{+},e)$ be an ordered *-vector space, with order unit $e \in V_{h}$. For each $ v \in V_{h}$, let
\begin{align*}
    \norm{v}:= \inf\{ r\geq 0: re \pm v \geq 0 \}
\end{align*}
and note that $ \normdot$ is a seminorm on $V_{h}$. We call $ \normdot$, the \textit{order seminorm}. If  $ (V, V^{+},e)$ is an AOU space,  i.e. the order unit $e$ is also Archimedean, then it was shown in 
\cite[Proposition 2.23]{10.2307/24903253} that $ \normdot$ is a norm, and can be extended to a norm on $V$.

Let $V$ and $W$ be two vector spaces. A linear map $ \phi :V \rightarrow W$, induces a linear map $ \phi^{(n)}: \Mn(V) \rightarrow \Mn(W)$ by setting $ \phi^{(n)}([v_{i,j}]) = [\phi(v_{i,j})] $. If $(V,V^{+},e)$ and $(W,W^{+},e')$ are two ordered *-vector spaces with order units $ e \in V_{h}$ and $ e' \in W_{h}$ then, a linear map $ \phi: V \rightarrow W$ is called \textit{unital} is $ \phi(e)= e'$.
Let $ (V,\{C_{n}\}_{n\in \N})$ and $(W,\{D_{n}\}_{n\in \N})$ be two matrix ordered *-vector spaces. A linear map $ \phi:V \rightarrow W$, is called 
\begin{itemize}
    \item \textit{positive}, if $ \phi(C_{1}) \subseteq D_{1}$
    \item \textit{n-positive}, if $ \phi^{(n)}(C_{n}) \subseteq D_{n}$, for some $n\in \N$,
    \item \textit{completely positive}, if $ \phi^{(n)}(C_{n}) \subseteq D_{n}$ for every $ n\in \N$.
\end{itemize} 
We say that 
 $\phi$ is a \textit{complete order isomorphism (c.o.i)}, if $\phi$ is a completely positive bijection and $\phi^{-1}$ is completely positive. 
We call $\phi : V \rightarrow W $, a \textit{complete order embedding (c.o.e.)}, if $ \phi$ is a complete order isomorphism onto its range. 
A  \textit{u.c.p.} map  is a unital and completely positive map. 

Let $\Hh$ be a Hilbert space. We denote by $ \Bh$ the space of all bounded linear operators on $\Hh$. The direct sum of $n$-copies of the Hilbert space $\Hh$, is denoted by $\Hh^{(n)}$. We also make the identification $ \Mn(\Bh) \cong \B(\Hh^{(n)})$, so that $ \Mn(\Bh)$ inherits a norm and a $ \Cs$-algebra structure. A subspace $ \Ss \subseteq \Bh$ is called a \textit{(concrete) operator system} if $ \Ss= \Ss^{*}$ and $ I_{\Hh} \in \Ss$ (where $I_{\Hh}$ is the identity operator on $\Hh$). Then $\Ss$ is an ordered *-vector space with the involution and order structure it inherits from $ \Bh$. Moreover, $ I_{\Hh}$ is an Archimedean order unit on $\Ss$. Since $ \Mn(\Ss) \subseteq \Mn(\Bh)\cong \B(\Hh^{(n)})$, we have that $ \Mn(\Ss)$ is also an ordered *-vector space with the order structure it inherits from $ \B(\Hh^{(n)})$ and has the identity $ I_{\Hh^{(n)}}$ 
as an Archimedean order unit. Thus 
every concrete operator system $ \Ss \subseteq \Bh$ is also an (abstract) operator system in the sense of Definition \ref{absoperatorsyst}. The following celebrated result of Choi and Effros states that the converse is also true.
\begin{theorem}\cite[Theorem 4.4]{CHOI1977156}
If $ (V, \{C_{n}\}_{n\in \N}, e)$ is an abstract operator system, then there exist a Hilbert space $\Hh$, a concrete operator system $ \Ss \subseteq \Bh$ and a unital complete order isomorphism $ \Phi: V \rightarrow \Ss$.
\end{theorem}
By the above Theorem, we no longer distinguish between concrete and abstract operator system and simply call each one of them an operator system.

We will use the following definition of an operator system kernel, for more equivalent definitions of the operator system kernel, see \cite{Kavruk2010QuotientsEA}.

\begin{definition}
Let $\Ss$ be an operator system and $\Jj $ a subspace of $ \Ss$. We say that $ \Jj \subseteq \Ss$ is a kernel in $\Ss$, if there exists an operator system $\Tt$ and a unital completely positive map $\phi: \Ss \rightarrow \Tt$ such that $ \Jj=\ker \phi $.
\end{definition}

A selfadjoint subspace $ \Jj \subseteq \Ss$ of an operator system $\Ss$ is called an \textit{order ideal}, if $ q \in \Jj$ and $ 0 \leq p \leq q$  implies that $ p \in \Jj$.

If $ \Jj \subseteq \Ss$ is an order ideal, we may define for each $ n \in \N$
\begin{align*}
    D_{n} := \{ [s_{i,j} + \Jj] \in \Mn(\Ss/\Jj): \exists k_{i,j} \in \Jj  \; \text{ such that }  \;   [s_{i,j}] + [k_{i,j}] \in \Mn(\Ss)^{+} \},
\end{align*}
and note that $(\Ss/\Jj, \{D_{n} \}_{n\in \N})$ is a matrix ordered *-vector space, with matrix order unit $ 1+ \Jj$. However, the unit may fail to be Archimedean for this family of cones.

\begin{proposition}\cite[Proposition 3.4]{Kavruk2010QuotientsEA} \label{matrixconesquo}
Let $\Ss$ be an operator system, and $ \Jj \subseteq \Ss$ be a kernel. For each $n \in \N$ we define 
\begin{align*}
    C_{n} := \{ [s_{i,j} + \Jj] \in \Mn(\Ss/\Jj): \forall \e>0 \; \;  \e e_n + [s_{i,j} + \Jj]  \in D_{n}  \}.  
    \end{align*}  
    Then, $(\Ss/\Jj, \{C_{n} \}_{n\in \N})$ is a matrix ordered *-vector space with Archimedean matrix order unit $ e + \Jj$. Moreover, the quotient map $ q : \Ss \rightarrow \Ss/\Jj$ is completely positive.
\end{proposition}

\begin{definition}
Let $\Ss$ be an operator system, and $ \Jj \subseteq \Ss$ be a kernel. We call the operator system $(\Ss/\Jj, \{C_{n} \}_{n\in \N}, e+ \Jj)$ defined in \ref{matrixconesquo}, the quotient operator system.  If  $D_{1}= C_{1}$, we say that the kernel $\Jj$ is order proximinal and if $ D_{n}= C_{n}$ for each $ n \in \N$, we say that $ \Jj$ is completely order proximinal.  
\end{definition}

Let $ \Aa$ be a unital $\Cs$-algebra. A  vector space $V$ is called an $\Aa$-\textit{bimodule} if there exist bilinear maps 
\begin{align*}
    \Aa \times V  & \rightarrow V  &  V \times \Aa &  \rightarrow V  \\
    (a,x) & \mapsto a \cdot x  &
    (x,a) & \mapsto x \cdot a,
\end{align*}
such that the following associativity conditions hold for all $ x \in V$, and $ a, b \in \Aa$:
\begin{enumerate}
    \item $(a \cdot x) \cdot b = a \cdot (x \cdot b)$
    \item $(ab) \cdot x = a \cdot (b \cdot x)$
    \item $ x \cdot (ab)  = (x \cdot a) \cdot b $
    \item $ 1 \cdot x =x = x \cdot 1$
\end{enumerate}

Let V be a  *-vector space that is also an $\Aa$-bimodule.  We  set \begin{align*}
     [a_{i,j}] \cdot [x_{i,j}] = \Big[ \sum_{k=1}^{m}a_{i,k}\cdot x_{k,j}  \Big] \; \; \text{ and }  \; \; \; \;  [x_{i,j}] \cdot [b_{i,j}] = \Big[ \sum_{k=1}^{n}x_{i,k}\cdot b_{k,j}  \Big] 
\end{align*}
for all $[x_{i,j}] \in M_{m,n}(V) $, $ [a_{i,j}] \in  M_{k,m}(\Aa)$, $[b_{i,j}] \in  M_{n,l}(\Aa) $.
So that, $ \Mn(V)$ becomes a $\Mn(\Aa)$-bimodule.
\begin{definition}
Let $\Ss$ be an operator system and $\Aa$ be a unital $\Cs$-algebra.  We call $\Ss$ an operator $\Aa$-system if
\begin{enumerate}
    \item $\Ss$ is an $\Aa$-bimodule
    \item $ (a \cdot s)^{*} = s^{*} \cdot a^{*}$ 
    \item $a \cdot e =  e \cdot a$ 
    \item $[a_{i,j}] \cdot [s_{i,j}] \cdot [a_{i,j}]^{*} \in \Mn(\Ss)^{+} $
\end{enumerate}
 for all $[a_{i,j}] \in M_{n,m}(\Aa)$,  $ [s_{i,j}] \in M_{m}(\Ss)^{+}$,  $ s \in \Ss$ and $ a \in \Aa$. 
 
 Condition 4. will be referred to as the $ \Aa$-\textit{compatibility} of the family of matrix cones $ ( \Mn(\Ss)^{+})_{n\in \N}$.
\end{definition}

The following result is a Choi-Effros type representation theorem of operator $\Aa$-systems that we will frequently use in the sequel. It states that we can  represent both the operator $\Aa$-system $\Ss$ and the unital $\Cs$-algebra $\Aa$ on the same Hilbert space $\Hh$, where the module action corresponds to operator multiplication. 
 \begin{theorem}\cite[Corollary 15.13]{paulsen_2003} \label{asystemsrepresentation}
 Let $\Aa$ be a unital $\Cs$- algebra and $\Ss$ be an (abstract) operator $ \Aa$-system. There exists a Hilbert space $\Hh$, a unital complete order embedding $ \phi : \Ss \rightarrow \Bh$ and a unital $*$-homomorphism $ \pi : \Aa \rightarrow \Bh$, such that 
 \[ \phi(a \cdot s) = \pi(a) \; \phi(s), \] for all $a\in \Aa$ and $s \in \Ss$.
 \end{theorem}

 \begin{remarks}
 
 1) Suppose that $ 1 \in \Aa \subseteq \mathcal{S} \subseteq \Bh$, where $ \Ss$ is a concrete operator system and $\Aa$ is a $\Cs$-algebra such that $ \Aa \cdot \Ss \subseteq \Ss$. We then have that
$$\Ss \cdot \Aa = \Ss^{*} \cdot \Aa^{*} = (\Aa \cdot \Ss)^{*} \subseteq \Ss^{*}=\Ss.$$
Such an operator system is also an operator $\Aa$-system.

2) If $\Ss$ is an operator $\Aa$-system, then by Theorem \ref{asystemsrepresentation}, there exist a Hilbert space $\Hh$, a unital complete order embedding $ \phi : \Ss \rightarrow \Bh$  and a unital *-representation $ \pi : \Aa \rightarrow \Bh$ such that $ \phi (a \cdot s)= \pi(a) \phi(s) $. So, 
$$ \pi(a) =\pi(a) \phi(e)= \phi( a \cdot e) \in \phi(\Ss)   $$
for all $ a\in \Aa$, and $ \pi(1)= \phi(1 \cdot e) = \phi(e) = I_{\Hh}$ which means that $$I_{\Hh}\in  \pi(\Aa) \subseteq \phi(\Ss) \subseteq \Bh$$
and $ \phi(\Ss)$ is an operator $\pi(\Aa)$-system with module action given by the multiplication of operators in $\Bh$.

3)  If further, in the above setting we had that the  unital *-homomorphism $ \pi: \Aa \rightarrow \Bh$ was faithful, we could  identify $\Aa$ with $ \pi(\Aa)$ and $ \Ss $ with its image under $\phi$ and reduce to the situation described in Remark 1).
Note however, that the unital *-homomorphisms $\pi$ may fail to be faithful.
 \end{remarks}

 \begin{definition} \label{faithfulopAsys}
 Let $\Ss$ be an operator $\Aa$-system. Denote its module action $ \Aa \times \Ss \rightarrow \Ss$, by $ a \cdot s $, for  $ a\in \Aa$, $s \in \Ss$. We will say that $ \Ss$ is a \textbf{faithful} operator $\Aa$-system, if 
 \begin{align*}
     a \cdot e \neq 0, \quad \text{ for all } \quad a \in \Aa \setminus \{0\}.
 \end{align*}
 \end{definition}

\begin{remark}\label{concreteopasyst} 
Let $\Ss$  be a faithful operator $\Aa$-system. 
In this case  Theorem \ref{asystemsrepresentation} gives the following more "concrete" representation: 
there exists a Hilbert space $\Hh$, a unital complete order embedding  $\phi :\Ss \rightarrow \B(\Hh)$ and a faithful unital *-representation $ \pi:\Aa \rightarrow \B(\Hh)$, such that 
 \[ \phi(a \cdot s) = \pi(a) \; \phi(s), \] for all $a\in \Aa$ and $s \in \Ss$.
 
 When this is the case, we may omit the inclusions, write $ 1 \in \Aa \subseteq \Ss \subseteq \Bh$ and treat the module action as multiplication of operators.

\end{remark}

\begin{definition}
Let $ \Ss$ and $ \Tt$ be two operator $\Aa$-systems. A linear map $ \phi: \Ss \rightarrow \Tt$ is called an $ \Aa$-bimodule map, if for every $s \in \Ss$ and $ a_{1}, a_{2} \in \Aa$,
\begin{align*}
    \phi(a_{1} \cdot s \cdot a_{2} )= a_{1} \circ \phi(s) \circ a_{2},
\end{align*}
where we denote by $ \cdot$ the module action of $\Ss$ and by $ \circ$ the module action of $\Tt$. We may sometimes, when there is no confusion, denote the module actions of $\Aa$ on $ \Ss$ and $\Tt$  with the same symbol, or with no symbol at all.
\end{definition}
In order to avoid future confusion and to point out the fact that even if $\phi$ is unital, $\Aa$ might not be contained in $ \Tt$, we state the following equivalent formulation.
\begin{remark}
Let $\Ss$ and $\Tt$ be two operator $\Aa$-systems and   $ \phi: \Ss \rightarrow \Tt$ be an $\Aa$-bimodule map. Then there exist a Hilbert space $ \Hh$, a complete order embedding of $ \Tt $ into $ \Bh$ and  a unital *-homomorphism $\pi: \Aa \rightarrow \Bh$ such that for every $s \in \Ss$ and $ a_{1}, a_{2} \in \Aa$,
\begin{align*}
    \phi(a_{1}\cdot s \cdot a_{2})= \pi(a_{1})\; \phi(s) \; \pi(a_{2}).
\end{align*}
\end{remark}

In the sequel, we will frequently use the following results.

\begin{proposition}\cite[Exercise 4.3]{paulsen_2003} \label{bimoduleequivalence}
Let $\B$, $ \mathcal{C} \subseteq \Bh$ be two unital $\Cs$-algebras and $\Aa$ be a unital $\Cs$-subalgebra such that $1 \in \Aa \subseteq \B \cap \Cc $. Let also $\phi : \B \rightarrow \mathcal{C}$ be a completely positive map.
Then, $\phi$ is an $\Aa$-bimodule map if and only if $ \phi(a)= a\cdot \phi(1) $ for every $a \in \Aa$.

\end{proposition}

\begin{proposition}\cite[Exercise 7.4]{paulsen_2003}
Let $\Aa$ and $\B$ be two unital $\Cs$-algebras and $ \Ss$ be an operator system such that $1\in  \Aa\subseteq \Ss \subseteq \B$. Suppose that $ \Aa \subseteq \B(\Hh)$. If $ \phi: \Ss \rightarrow \B(\Hh)$ is a completely positive  $\Aa$-bimodule map, then every completely positive extension of $\phi$ to $\B$ is also an $\Aa$-bimodule map. 
\end{proposition}

We recall the categorical notion of the coproduct.
\begin{definition} \label{coprodcat}
Let $\mathcal{O}_{1}$ and $\mathcal{O}_{2}$ be two objects in a category $\mathcal{G}$. Their coproduct, is another object  $\mathcal{O}_{1}*\mathcal{O}_{2}$, also in the category $ \mathcal{G}$, together with two morphisms $ \phi_{1}:\mathcal{O}_{1} \rightarrow \mathcal{O}_{1}*\mathcal{O}_{2} $   and $ \phi_{2}:\mathcal{O}_{2} \rightarrow \mathcal{O}_{1}*\mathcal{O}_{2}$, that satisfies the following universal property:
If $ \mathcal{O}$ is an object in the same category with morphisms $ \psi_{1}:\mathcal{O}_{1} \rightarrow \mathcal{O}$ and $ \psi_{2}:\mathcal{O}_{2} \rightarrow \mathcal{O}$, then there exists a unique morphism $ \Psi: \mathcal{O}_{1}*\mathcal{O}_{2} \rightarrow \mathcal{O}$ such that $ \Psi \circ \phi_{1}= \psi_{1}$ and $ \Psi \circ \phi_{2}= \psi_{2}$. That is, the following diagram commutes:
 \begin{center}
          \begin{tikzcd}[row sep=large, column sep=large]
             \mathcal{O}_{1}  \arrow[rd, "\psi_{1}"]   \arrow[r, "\phi_{1}"]  & \mathcal{O}_{1}*\mathcal{O}_{2} \arrow[d, "\Psi"] & \mathcal{O}_{2} \arrow[ld, "\psi_{2}"] \arrow[l, "\phi_{2}"] \\
             & \mathcal{O}  
              \end{tikzcd}
                \end{center}
\end{definition}

It is known that the coproduct of two operator systems exists and can be realised as an operator subsystem of the free product of the $\Cs$-algebras they are contained in \cite[Theorem 5.2]{bbf61d43a5de485bad4fe7b6c908f9a2}.
As we will show, the same can be said about the coproducts of a certain class of operator $ \Aa$-systems, only this time, they are realised as  subsystems of the  free product of the $\Cs$-algebras,  amalgamated over $\Aa$.

Let $\Aa_{i}$, $i=1, 2$  be two unital $\Cs$-algebras. We say that $\Aa_{1}$ and $\Aa_{2}$ contain a common unital  $\Cs$-subalgebra  $ \B$, if there exist unital embeddings $ \e_{i}: \B \rightarrow \Aa_{i}$ for $ i=1,2$.

\begin{theorem}\cite[Davidson-Kakariadis version of Boca's Theorem]{davidson_kakariadis_2019} \label{DKversion}
Let $\Aa_{i}$, $i=1, 2$  be two unital $\Cs$-algebras containing a common unital $ \Cs$-subalgebra $ \B$. Let  $ \Phi_{i} : \Aa_{i} \rightarrow \Bh$, $ i=1, 2$ be two unital completely positive maps which restrict to a common linear map of $\B$. Then, there exists a unital completely positive map $ \Phi:=\Phi_{1} * \Phi_{2}: \Aa_{1} *_{\B} \Aa_{2} \rightarrow \Bh$ such that $ \Phi\arrowvert_{\Aa_{i}} = \Phi_{i}$ for $ i=1, 2$.
\end{theorem}

\section{EXISTENCE OF OPERATOR A-SYSTEM  COPRODUCTS}
Suppose that we work in the category whose objects are operator $\Aa$-systems and whose morphisms are unital completely positive $\Aa$-bimodule maps. A coproduct of two operator $\Aa$-systems $ \Ss$ and $ \Tt$,  would be an operator $\Aa$-system $X$ together with two ucp $\Aa$-bimodule maps $ \phi_{1}:\Ss \rightarrow X$ and $ \phi_{2}:\Tt \rightarrow X$ such that the universal property for coproducts holds. However we will restrict to the class of operator $\Aa$-systems that are faithful in the sense of definition \ref{faithfulopAsys} and prove the existence of the coproduct in this case.   
\begin{remark}
   The faithfulness assumption cannot be omitted in general. This is clearly seen in the following example.\footnote{Thanks to Prof. Mihalis Anoussis.} 
\end{remark}
\begin{proof}
Let $\Aa$ be the unital $\Cs$-algebra  $\bb C \oplus \bb C$; let $\Ss$ and $\Tt$  be the one-dimensional operator $\Aa$-systems with actions defined for 
$a=(a_1,a_2)\in \Aa$ by
$$a\cdot s := a_1s, \; \; s\in \Ss, \qquad a\cdot t := a_2t, \; \;  t\in \Tt. $$
Suppose that $X$ is an operator $\Aa$-system satisfying the universal property 
of Definition \ref{coprodcat}. Then there exist ucp $\Aa$-bimodule maps $\phi_1 : \Ss \to X$ 
and $\phi_2 : \Tt \to X$ so that
$$\phi_1(a \cdot s) = a \cdot \phi_1 (s), \qquad 
\phi_2(a \cdot t) = a \cdot \phi_2 (t)$$
where $s\in \Ss, t \in \Tt$ and $a\in \Aa$.

Observe that if $a=(0,a_2)\in \Aa$ then $a\cdot e_S=0$ and so, since $\phi_1$ is unital, 
$$0=\phi_1(a\cdot e_S)=a\cdot e_X\, .$$
Similarly if $a'=(a_1,0)\in \Aa$ then $a'\cdot e_T=0$ and so
$$a'\cdot e_X=0\, .$$
It follows that for every $a=(a_1,a_2)\in \Aa$ we have
$$a\cdot e_X=(a_1,0)\cdot e_X+(0,a_2)\cdot e_X=0$$
which is a contradiction because the unit of $\Aa$ acts as an identity on $X$, i.e. $1_{\Aa} \cdot x = x$ for every $ x \in X$. 
\end{proof}

In the example above, it becomes apparent that for some operator $\Aa$-systems, not only does their  coproduct  not exist, but they don't even admit ucp $\Aa$-bimodule maps into the same operator $\Aa$-system. The referee also pointed out the following characterisation, regarding the existence of the coproduct. We give a sketch of the proof.

\begin{proposition} \label{refereesprop}
Let $\Aa$ be a C*-algebra and $\Ss_{1}, \Ss_{2}$ be two operator $\Aa$-systems. Then, the coproduct $\Ss_{1}\oplus_{\Aa}\Ss_{2}$ exists if and only if there is a triplet $(\phi,\psi,\Ss)$ where $\Ss$ is an operator $\Aa$-system and $ \phi : \Ss_{1} \rightarrow \Ss$ and $\psi : \Ss_{2} \rightarrow \Ss$ are ucp $\Aa$-bimodule maps.
\end{proposition}
\begin{proof}

The necessity is clear. For the converse; for each "admissible" triplet  $(\phi,\psi,\Ss)$ we call $\Hh_{(\phi,\psi,\Hh)} $  the Hilbert space on which $\Ss$ acts and then take the direct sum over "sufficiently many" such Hilbert spaces. That is, 
\begin{align*}
     \Hh':= \oplus_{(\phi,\psi,\Ss)} \Hh_{(\phi,\psi,\Hh)},  
\end{align*}
 (see for example \cite[Proposition 2.4.2]{timoney_2007} for related constructions).
Also define the maps
\begin{align*}
    \e_{1}:= \oplus_{(\phi,\psi,\Ss)}\phi: \Ss_{1}&  \rightarrow \B(\Hh')\\
     s_{1}& \mapsto \oplus_{(\phi,\psi,\Ss)}\phi(s_{1})
\end{align*}
and
\begin{align*}
    \e_{2}:= \oplus_{(\phi,\psi,\Ss)}\psi: \Ss_{2}&  \rightarrow \B(\Hh')\\
     s_{2}& \mapsto \oplus_{(\phi,\psi,\Ss)}\psi(s_{2}).
\end{align*}
 Now we set $ \mathcal{S}:= \e_{1}(\Ss_{1}) + \e_{2}(\Ss_{2}) $ to be the subsystem of $ \B(\Hh')$ generated by $\e_{1}$ and $ \e_{2}$. Then, it is not hard to show that $ \mathcal{S}$ is in fact an operator $\Aa$-system that satisfies the universal property of the coproduct, with corresponding maps $\e_{1}, \e_{2}$.
\end{proof}

Next, we  prove the existence of the coproduct, for faithful operator $\Aa$-systems.  In fact, what is shown is the existence of a triplet $ (\phi,\psi,\Ss)$ satisfying the requirements of Proposition \ref{refereesprop}. We also provide a more explicit construction of the coproduct as a subsystem of the amalgamated free product of $\Cs$-algebras.

\begin{theorem} \label{existenceofcop}
Let $\Ss_{1}$ and $ \Ss_{2} $ be two faithful operator $\Aa$-systems. Then, their coproduct exists, is denoted by $\Ss_{1} \oplus_{\Aa} \Ss_{2}$, it is a faithful operator $\Aa$-system and is unique up to a complete order isomorphism that is also an $\Aa$-bimodule map.
\begin{proof}
Let $ \Ss_{1} \subseteq \B(\Hh_{1})$ and $ \Ss_{2} \subseteq \B(\Hh_{2})$ be the operator $ \Aa$-systems, where $ \Aa$ is a unital $ \Cs$-algebra, $ \Hh_{i}$ be two Hilbert spaces and $ \pi_{i}:\Aa \rightarrow \B(\Hh_{i})$, $i=1,2$ the respective faithful representations.

Let  $\B(\Hh_{1}) *_{\Aa} \B(\Hh_{2}) $ denote the amalgamated free product over $ \Aa$ \cite[Theorem 3.1]{10.2307/24892242}. This is the  $\Cs$-algebra generated by $i_{1}(\B(\Hh_{1})) \cup i_{2}(  \B(\Hh_{2}))$ for  embeddings $ i_{k} :\B(\Hh_{k}) \hookrightarrow \B(\Hh_{1}) *_{\Aa} \B(\Hh_{2})  $, $k =1,2 $ such that $ i_{1}\circ \pi_{1} = i_{2} \circ \pi_{2}$ that satisfies the following property: \\
Whenever $ \phi_{k} : \B(\Hh_{k}) \rightarrow \Bh$, $ k=1,2$ are *-representations such that $\phi_{1}\circ \pi_{1}= \phi_{2}\circ\pi_{2}$ there exists a *-representation $ \pi : \B(\Hh_{1}) *_{\Aa} \B(\Hh_{2})\rightarrow \Bh $ such that $ \pi \circ i_{k} = \phi_{k}$.\\
We define  
\begin{align*}
    \Ss_{1} + \Ss_{2}:= i_{1}(\Ss_{1}) + i_{2}(\Ss_{2}) \subseteq  \B(\Hh_{1}) *_{\Aa} \B(\Hh_{2})
\end{align*} 
which is an operator subsystem of the amalgamated free product and is such that the inclusions of $ \Ss_{1}, \Ss_{2}$ in it are complete order embeddings.
Since $i_{k}, \pi_{k} $, $ k=1,2 $ are injective *-homomorphisms thus completely isometric, we identify $ \Aa \cong i_{1}\circ \pi_{1}(\Aa) = i_{2}\circ \pi_{2}(\Aa) $ and $i_{k}(\B(\Hh_{k})) \cong \B(\Hh_{k}) $, $ k=1,2 $. So we omit the embeddings and consider $\B(\Hh_{1}) $ and $ \B(\Hh_{2})$   as $\Cs$-subalgebras of $\B(\Hh_{1}) *_{\Aa} \B(\Hh_{2}) $, having $\Aa$ as a common $\Cs$-subalgebra and write \begin{align*}
  \Ss_{1} + \Ss_{2}=  \{ s_{1}+ s_{2} :s_{1}\in \Ss_{1}, s_{2} \in \Ss_{2} \} \subseteq  \B(\Hh_{1}) *_{\Aa} \B(\Hh_{2}).
\end{align*}

Then, $\Ss_{1} + \Ss_{2}$ is an operator system, and also an $ \Aa$-bimodule that contains $\Aa$ as subsystem, with module action given by
\[ a \cdot (s_{1} + s_{2} ) = a \cdot s_{1} + a \cdot s_{2}, \; \; \; \; a \in \Aa, \; s_{1} \in \Ss_{1}, \; s_{2} \in \Ss_{2}. \]
The module action is in fact multiplication of elements in the free product $\Cs$-algebra, so it is well defined and the family of matrix cones is $\Aa$-compatible, that is, $ \Ss_{1} + \Ss_{2}$ is a faithful operator $\Aa$-system.

We will prove that  $\Ss_{1} + \Ss_{2}$, with the associated maps being the inclusions, has the desired universal property of the coproducts in the category of operator $\Aa$-systems.
Suppose that we have another operator $\Aa$-system $ \Tt \subseteq \B(\Kk)$, where $ \rho : \Aa \rightarrow \B(\Kk)$ is the associated representation, and consider two u.c.p. $ \Aa$-bimodule maps,
\begin{align*}
    \psi_{1}: \Ss_{1} \rightarrow \Tt \\
    \psi_{2}: \Ss_{2} \rightarrow \Tt. 
\end{align*}
By Arveson's theorem, we extend them to u.c.p. maps: 
\begin{align*}
    \Tilde{\psi}_{1}: \B(\Hh_{1}) \rightarrow \B(\Kk) \\
    \Tilde{\psi}_{2}: \B(\Hh_{2}) \rightarrow \B(\Kk). 
\end{align*}
These two are unital completely positive maps that agree on $\Aa$ with a common *-representation, indeed,
$$ \Tilde{\psi}_{1}(a)= \psi_{1}(a)=\rho(a)=\psi_{2}(a)= \Tilde{\psi}_{2}(a), \; \; \; \; \; a\in \Aa.$$
By Boca's Theorem (\ref{DKversion}),  there exists a unital completely positive map 
\[ \Psi : \B(\Hh_{1}) *_{\Aa} \B(\Hh_{2}) \rightarrow \B(\Kk) \] such that $ \Psi\arrowvert_{\B(\Hh_{i})} = \Tilde{\psi}_{i}$ for $ i =1,2$ and of course $ \Psi(a) =\rho(a)$ for all $ a \in \Aa$.
Now define $ \Phi := \Psi\arrowvert_{\Ss_{1}+\Ss_{2}}$, so that 
\[ \Phi : \Ss_{1} + \Ss_{2} \rightarrow \B(\Kk) \]
is a unital completely positive map such that for every $ s_{1} \in \Ss_{1}, \; \; s_{2} \in \Ss_{2}$:
\begin{align*}
    \Phi(s_{1} + s_{2}) & = \Psi\arrowvert_{\Ss_{1}+\Ss_{2}} (s_{1} + s_{2})\\
    &= \Psi\arrowvert_{\Ss_{1}+\Ss_{2}} (s_{1}) +\Psi\arrowvert_{\Ss_{1}+\Ss_{2}} (s_{2})\\
    &= \Tilde{\psi}_{1}(s_{1}) +\Tilde{\psi}_{2}(s_{2})\\
    & = \psi_{1}(s_{1}) + \psi_{2}(s_{2}) \in \Tt.
\end{align*}
This means that $\Phi$ has its image inside $ \Tt$ and also that $ \Phi\arrowvert_{\Ss_{i}}= \psi_{i}$ for $ i=1,2$.
Moreover, 
\begin{align*}
    \Phi(a \cdot (s_{1} + s_{2}))& = \Phi(a \cdot s_{1} + a \cdot s_{2})\\
    &=  \psi_{1}(a \cdot s_{1}) + \psi_{2}(a \cdot s_{2})\\
    &= \rho(a) \cdot \psi_{1}(s_{1}) + \rho(a) \cdot \psi_{2}(s_{2}) \\
     & = \rho(a) \cdot ( \psi_{1}(s_{1}) + \psi_{2}(s_{2}))\\
     & = \rho(a) \cdot \Phi(s_{1}+ s_{2})
\end{align*}
for every $ a \in \Aa, \; \; s_{1} \in \Ss_{1}, \; \; s_{2} \in \Ss_{2} $. That is, $ \Phi$ is a unital completely positive, $ \Aa$-bimodule map with image inside $ \Tt$ and whose restrictions to $ \Ss_{1}$ and $ \Ss_{2}$ gives $ \psi_{1}$ and $ \psi_{2}$ respectively. Thus, the operator $\Aa$-system  $  \Ss_{1} + \Ss_{2}  $  satisfies the universal property for coproducts of operator $\Aa$-systems. Finally, it is easy to check that two operator $\Aa$-systems that satisfy this universal property are completely order isomorphic. 
\end{proof}
\end{theorem}

In the proof of Theorem \ref{existenceofcop}, we identify the coproduct of two operator $\Aa$-systems $\Ss$  and $ \Tt$, as a subsystem of the amalgamated free product of their containing operator spaces $\Bh$ and $\Bk$ respectively. 
In fact, we can consider the coproduct of two operator $\Aa$-systems $\Ss$ and $\Tt$ as a subsystem of amalgamated free product of any two $ \Cs$-algebras that contain $\Ss$ and $ \Tt$, as long as they respect their module actions.
\begin{proposition} \label{intheircsalgebras}
Let $ \Aa$, $\B_{1}, \B_{2}$ be  unital $\Cs$-algebras and $ \Ss_{1} \subseteq \B_{1} $,   $ \Ss_{2} \subseteq \B_{2}$ be two operator $ \Aa$-systems such that there exists  injective *-homomorphisms $\pi_{i}: \Aa \rightarrow \B_{i}$  with $ s_{i} \cdot a = s_{i} \pi_{i}(a)$ for every $ i =1, 2$, $ s_{i} \in \Ss_{i}$ and $ a \in \Aa$. Let also
\begin{align*}
    \Ss_{1}+\Ss_{2}:= \{s_{1}+s_{2}:s_{1}\in \Ss_{1}, s_{2} \in \Ss_{2}\} \subseteq \B_{1} *_{\Aa}\B_{2}.
\end{align*}
Then,  $\Ss_{1} \oplus_{\Aa}\Ss_{2} \cong \Ss_{1}+\Ss_{2}$ by a unital complete order isomorphism, associated with the inclusions $ \Ss_{i} \hookrightarrow \B_{i}$, $i=1,2$.
\begin{proof}
  The proof is essentially the same as in Theorem \ref{existenceofcop}. Let $\Tt \subseteq \B(\Hh)$ be an operator $\Aa$-system and $ \phi_{i}: \Ss_{i} \rightarrow \Tt$, $i=1,2$ be two u.c.p. $\Aa$-bimodule maps. We extend them to two unital completely positive $ \Aa$-bimodule maps $ \Tilde{\phi}_{i}: \B_{i} \rightarrow \B(\Hh) $, $ i=1,2$. By Boca's theorem \ref{DKversion}, there exists a u.c.p. $\Aa$-bimodule map $ \Phi : \B_{1} *_{\Aa}\B_{2} \rightarrow \B(\Hh) $ such that $ \Phi\arrowvert_{\B_{i}} = \Tilde{\phi}_{i}$, for $ i=1,2$. Finally, the map $\Phi\arrowvert_{\Ss_{1}+\Ss_{2}}$ has the desired properties and has image inside $ \Tt$.
\end{proof}
\end{proposition}

\begin{definition}[Coproduct of operator $\Aa$-systems] \label{defAcoprod}
Let $ \Ss_{1}$ and $ \Ss_{2}$ be two faithful operator $ \Aa$-systems. The coproduct of $\Ss_{1}$ and $\Ss_{2}$ is the unique faithful operator $ \Aa$-system $\Ss_{1} \oplus_{\Aa} \Ss_{2}$, along with unital complete order embeddings $\phi_{i}:\Ss_{i} \hookrightarrow \Ss_{1} \oplus_{\Aa} \Ss_{2} $, $i= 1, 2$ that are also $\Aa$-bimodule maps,  such that the following universal property holds: For every operator $\Aa$-system $ \Rr $ and u.c.p. $\Aa$-bimodule maps $ \psi_{i} : \Ss_{i} \rightarrow \Rr$, $ i=1, 2$, there exists a unique u.c.p. $\Aa$-bimodule map $ \Psi :\Ss_{1} \oplus_{\Aa} \Ss_{2} \rightarrow \Rr $ such that $ \Psi \circ \phi_{i}= \psi_{i}$ for $ i =1, 2$. 
\end{definition}
 It is a standard consequence of the universal property, that the coproduct of two operator $\Aa$-systems is unique up to a complete order isomorphism.
Theorem \ref{existenceofcop}  proves the existence of operator $\Aa$-system coproducts but also  generalises the operator system case  in  \cite[Theorem 5.2.]{bbf61d43a5de485bad4fe7b6c908f9a2}. 
Furthermore, as in the operator system case \cite[Section 8]{Kavruk2014}, we can give a more concrete realisation of this coproduct, in terms of quotients of operator systems.

\begin{remark}\label{sameembedding}
Suppose that we have two faithful operator $\Aa$-systems $ \Ss$ and $ \Tt$, where $\Aa$ is a unital $ \Cs$-algebra. Then, by arguing as in the proof of Theorem \ref{existenceofcop}, we may assume that  $1\in  \Aa \subseteq \Ss \cap \Tt \subseteq \Bh$ where $ \Ss$ and $\Tt$ are $\Aa$-bimodules, with module action given by multiplication of operators acting on the Hilbert space $\Hh$.
\end{remark}

\begin{theorem} \label{theorimakernel}
Let $\Aa$ be a unital $\Cs$-algebra, $ \Ss$ and $ \Tt$ be two faithful operator $ \Aa$-systems, and \[ \mathcal{J}:= \{ a\oplus (-a) : a \in \Aa\} \subseteq \Ss \oplus \Tt. \]
Then, there is an operator $\Aa$-system structure on $\Ss \oplus \Tt/ \mathcal{J} $ and up to a complete order isomorphism 
\[ \Ss \oplus \Tt/ \mathcal{J} \cong  \Ss \oplus_{\Aa} \Tt. \]
Moreover, the quotient map $ q:\Ss\oplus \Tt \rightarrow \Ss \oplus \Tt/ \mathcal{J}$ is a u.c.p. $\Aa$-bimodule map.
\begin{proof}
By Remark \ref{sameembedding}, we assume that $1\in  \Aa \subseteq \Ss \cap \Tt \subseteq \Bh$, where $ \Hh$ is a Hilbert space.
   Let 
 $ \Ss \oplus \Tt \subseteq \B(\Hh \oplus \Hh) $
be their direct sum and define
\[ \mathcal{J}:= \{ a\oplus (-a) : a \in \Aa\} \subseteq \Ss \oplus \Tt. \]
We will prove that $\Ss\oplus \Tt / \Jj$:
\begin{enumerate}
    \item is an operator system.
    \item  is an operator $\Aa$-system
    \item satisfies the universal property for operator $ \Aa$- system coproducts.
\end{enumerate}
1. \; Note that $\mathcal{J} $ is a closed, selfadjoint subspace of $ \Ss \oplus \Tt$ that  does not contain any  positive elements other than zero, so it is trivially  an order ideal.  For each $n \in \N$, we set
\begin{align*}
     D_{n} := \{ [s_{i,j}\oplus t_{i,j} +\Jj]\in \Mn(\Ss\oplus \Tt /\Jj): \exists \; a_{i,j} \in \Aa  \text{ with } [s_{i,j}]\oplus[ t_{i,j}] + [a_{i,j}]\oplus  [-a_{i,j}] \geq 0  \},
\end{align*}
and also 
\begin{align*}
    C_{n}:=\{[s_{i,j}\oplus t_{i,j} +\Jj]\in \Mn(\Ss\oplus \Tt /\Jj): \e (1\oplus 1)\otimes I_{n} +[s_{i,j}\oplus t_{i,j}+ \Jj]  \in D_{n} \;   \text{ for all } \e >0  \}.
\end{align*}
We will prove that $ (\Ss \oplus \Tt /\Jj, (C_{n})_{n\in \N}, 1\oplus 1 + \Jj )$ is a matrix ordered *-vector space and $ 1\oplus 1 + \Jj$ is an Archimedean matrix order unit. By \cite[Proposition 3.16]{tomforde_paulsen}, it suffices to prove that $ \Jj$ is a kernel, which by \cite[Lemma 3.3]{Kavruk2010QuotientsEA} is equivalent to saying that the order seminorm on $ \Ss \oplus \Tt / \Jj$ is a norm.  

\textit{Claim}: If $ x\oplus -x + \Jj \in C_{1}$ then  $ x \oplus -x \in \Jj$, i.e. $ x \in \Aa$. Indeed;  for every $ \e >0$ there exists an $ a_{\e} \in \Aa_{h}$ such that $ \e(1\oplus1) + x \oplus -x + a_{\e}\oplus -a_{\e} \geq 0$.  So, for every $ n \geq 1$, we can chose $ a_{n}\in \Aa_{h}$ so that  
\begin{align*}
    \frac{1}{n}1 + x + a_{n} \geq 0\\
    \frac{1}{n}1 -x - a_{n} \geq 0.
\end{align*}
So, $ -\frac{1}{n}1 \leq x + a_{n} \leq \frac{1}{n}1$. Thus,  (note that $x $ is hermitian) there exists a sequence $ (a_{n})_{n \in \N}\subseteq \Aa$ such that for every $ n \in \N$ it holds that $ \norm{x + a_{n}}\leq \frac{1}{n}$. Hence  $ -x = \lim_{n \to +\infty}a_{n} $ is in $\Aa$, thus $ x \in \Aa$ and consequently $ x \oplus -x \in \Jj$.

Now let $ u \in (\Ss \oplus \Tt/\Jj)_{h}$ and suppose that $ \norm{u}=0$, where $ \normdot$ is the order seminorm.  So
\begin{align*}
   \e (1\oplus1 + \Jj) + u \in D_{1}   \; \text{ and } \; \e (1\oplus1 + \Jj) - u \in D_{1}
\end{align*}
for all $ \e>0$. Write $ u= s\oplus t + \Jj$, so we have equivalently that for all $\e>0$, there exists $ a_{\e}, b_{\e} \in \Aa_{h}$ such that 
\begin{align*}
    \e(1\oplus1) + s\oplus t + a_{\e}\oplus-a_{\e} \geq 0 \; \text{ and } \;     \e(1\oplus1) - s\oplus t + b_{\e}\oplus-b_{\e} \geq 0.
\end{align*}
Equivalently $\forall \e>0, \; \; \exists a_{\e}, b_{\e} \in \Aa_{h}$ such that 
\begin{align}
    \e1 + s +a_{\e} \geq 0 \label{1}\\
    \e1 + t -a_{\e} \geq 0 \label{2}
\end{align}
and also 
\begin{align}
    \e1 - s +b_{\e} \geq 0 \label{3} \\
    \e1 - t -b_{\e} \geq 0 \label{4}.
\end{align}
Now, by adding relations \ref{1} and \ref{2} we get that $\e 1+ s+t \geq 0$ for all $ \e >0$ and therefore $ s+t \geq 0$ while if we add \ref{3} and \ref{4} we get that $ -(s+t) \geq 0$. This means that $ s+t=0$, that is, $ t=-s$.  

So,  the assumption that $ \e (1\oplus1 + \Jj) + u \in D_{1} $ for all $\e>0$, is equivalent to $ s\oplus -s +\Jj \in C_{1}$, which by the aforementioned claim implies that $ s\oplus -s \in \Jj$. This means that the order seminorm is a norm, and thus $ \Jj$ is a kernel.

2. \; Let  $ q :  \Ss \oplus \Tt \rightarrow \Ss \oplus \Tt/ \mathcal{J}  $ be the quotient map, and define the left module action by
\begin{align*}
    \Aa \; \times \;  (\Ss \oplus \Tt/ \mathcal{J} ) &\rightarrow  \Ss \oplus \Tt/ \mathcal{J}  \\
    (a, q(s_{1}\oplus s_{2}))  &\mapsto q((a\cdot s_{1}) \oplus( a \cdot s_{2})),
\end{align*}
which we denote by $ a \diamond q(s_{1}\oplus s_{2}):= q((a\cdot s_{1}) \oplus( a \cdot s_{2})) $. One may define, in a similar manner, the right module action and denote it by  $   q(s_{1}\oplus s_{2})\diamond a:= q(( s_{1}\cdot a) \oplus( s_{2} \cdot a)) $.

-This is well defined. Indeed if $s_{1},s_{1}' \in \Ss$ and $ s_{2},s_{2}' \in \Tt$ are such that $$ s_{1}\oplus s_{2} + \mathcal{J}=s_{1}'\oplus s_{2}' + \mathcal{J}, $$
i.e. $(s_{1}\oplus s_{2}) - (s_{1}'\oplus s_{2}') \in \mathcal{J} $, then there exists an $ a_{0} \in \Aa$ s.t. $s_{1} - s_{1}' = a_{0} $ and $ s_{2}- s_{2}' = -a_{0}$. But then, for all $a\in\Aa$ we have $as_{1} - as_{1}' = aa_{0} $ and $ as_{2}- as_{2}' = -aa_{0}$, or equivalently 
 $(as_1\oplus as_2) - (as_1'\oplus as_2') \in \mathcal{J} $ (since $aa_0\in\Aa$), that is 
 $$ a\diamond(s_1\oplus s_2 + \mathcal{J})
 =a\diamond (s_1'\oplus s_2' + \mathcal{J}). $$
 
- It is also easy to verify that $ (a \diamond q(s\oplus t))^{*} = q(s\oplus t)^{*} \diamond a^{*}$.

- It remains to check that the family of matrix cones $(C_{n})_{n\in \N}$ on $ \Mn(\Ss\oplus \Tt/\Jj)$ are  $\Aa$-compatible. Note first that the family $ (D_{n})_{n\in \N}$ is  $ \Aa$-compatible; indeed,  if $[s_{i,j}\oplus t_{i,j} +\Jj]\in D_{n}$ and $ B \in M_{m,n}(\Aa)$ then there exist $ A \in \Mn(\Aa)$ such that $ [s_{i,j}\oplus t_{i,j}] +A \oplus -A\geq 0 $. So, $ B \diamond [s_{i,j}\oplus t_{i,j} +\Jj] \diamond B^{*} \in D_{m} $ since  we easily see that
\begin{align*}
  ( B \cdot S \cdot B^{*} )\oplus 
  (B \cdot T \cdot B^{*}) +    (B \cdot A \cdot B^{*} )\oplus-( B \cdot  A \cdot B^{*}) \geq 0,
\end{align*}
where $ S = [s_{i,j}]$, $ T= [t_{i,j}]$ and  $B\cdot  A \cdot B^{*} \in M_{m}(\Aa) $.

Now let $[s_{i,j}\oplus t_{i,j} +\Jj]\in C_{n}$ and $ B \in M_{m,n}(\Aa)$. Note  that $B \diamond (1\oplus 1 + \Jj)_{n} \diamond B^{*} \in \Mn(\Ss\oplus \Tt/\Jj) $ is selfadjoint and since $ 1 \oplus 1 + \Jj$ is a matrix unit for the cones $ (D_{n})_{n\in \N}$, there exists an $ \e_{0} >0$ such that  
\begin{align*}
    \e_{0} (1\oplus1 +\Jj)_{m} - B \diamond (1\oplus 1 + \Jj)_{m} \diamond B^{*} \in D_{m}  
\end{align*}
equivalently
\begin{align*}
     \e_{0}(I_{m}\oplus I_{m}) - BB^{*}\oplus BB^{*} + M_{m}(\Jj) \in D_{m}.
\end{align*}
Since $[s_{i,j}\oplus t_{i,j} +\Jj]\in C_{n}$ we have that for all $ \e >0$  
\begin{align*}
    \frac{\e}{\e_{0}}(I_{n}\oplus I_{n}) + [s_{i,j}]\oplus [t_{i,j}] + \Mn(\Jj) \in D_{n}.
\end{align*}
So by the $ \Aa$-compatibility of the $D_{n}$, we have that 
\begin{align*}
    B \diamond \big( \frac{\e}{\e_{0}}(I_{n}\oplus I_{n}) + [s_{i,j}]\oplus [t_{i,j}] \big) \diamond B^{*} + M_{m}(\Jj) \in D_{m} 
    \end{align*}
    and therefore
    \begin{align*}
    \frac{\e}{\e_{0}} (BB^{*}\oplus BB^{*}) + ( B\cdot [s_{i,j}] \cdot B^{*})\oplus (B \cdot [t_{i,j}] \cdot B^{*}) + M_{m}(\Jj) \in D_{m}.
\end{align*}
Finally, if we add $\frac{\e}{\e_{0}} (\e_{0}(I_{m}\oplus I_{m}) - BB^{*}\oplus BB^{*} )$ to the latter relation, we get that 
\begin{align*}
    \e(I_{m}\oplus I_{m}) +  ( B\cdot [s_{i,j}] \cdot B^{*})\oplus (B \cdot [t_{i,j}] \cdot B^{*}) + M_{m}(\Jj) \in D_{m} 
\end{align*}
equivalently,
\begin{align*}
    \e (I_{m}\oplus I_{m}) + B \diamond ( [s_{i,j}]\oplus [t_{i,j}]) \diamond B^{*} + M_{m}(\Jj) \in D_{m},
\end{align*}
that is, $B \diamond [s_{i,j}\oplus t_{i,j} +\Jj] \diamond B^{*}\in C_{m}$, which proves the $\Aa$-compatibility of $ (C_{n})_{n\in \N}$.  

3. \; So $\Ss \oplus \Tt/ \mathcal{J} $ is an operator $ \Aa$-system, and it remains to prove that it satisfies the universal property for the coproducts. To this end, we begin with defining the maps \\
\begin{align*}
    i_{1}: \Ss \rightarrow \Ss \oplus \Tt/ \mathcal{J} \; \; \text{ and } \; \;  i_{2}: \Tt \rightarrow \Ss \oplus \Tt/ \mathcal{J} 
\end{align*}
 by 
\begin{align*}
    i_{1}(u_{1}) = 2q(u_{1}\oplus 0), \; \; \; \; \; i_{2}(u_{2})= 2 q(0 \oplus u_{2}), \; \; \; \; \; \; u_{1}\in \Ss, \; u_{2} \in \Tt.
\end{align*}
These are completely positive because the quotient map is completely positive. 
In fact, they are complete order isomorphisms. Indeed, because if $i_{1}^{(n)}([u_{i,j}]) = 2 [u_{i,j}]\oplus 0 + \Jj \in C_{n}$, then for every $ \e>0$ there exist $ [a_{i,j}] \in \Aa$ s.t. $ \e (1\oplus1) \otimes I_{n} + 2 [u_{i,j}]\oplus 0 + [a_{i,j}]\oplus -[a_{i,j}] \geq 0$. This implies that $ \e I_{n} + 2 [u_{i,j}] + [a_{i,j}] \geq 0 $ and $ \e I_{n} - [a_{i,j}] \geq 0$, so by adding those, $ 2\e I_{n} + 2 [u_{i,j}] \geq 0$, and since $ \e >0$ was arbitrary, $ [u_{i,j}] \geq 0$.
Moreover
\[ i_{1}(1) = 2 q(1 \oplus 0) = q(1 \oplus 1) = 1  \]
and the same holds for $ i_{2}$. Also, for all $ a \in \Aa$,
\[ i_{1}(a) = 2 q(a \oplus 0) = q(a \oplus a)= 2 q(0 \oplus a) = i_{2}(a),    \]
that is, $ i_{1}\arrowvert_{\Aa} = i_{2}\arrowvert_{\Aa}$ and the inclusion  $ \Aa \hookrightarrow  \Ss \oplus \Tt/ \mathcal{J} $ is well defined. Also, the inclusion maps are $ \Aa$-bimodule maps:
\[ i_{1}(a \cdot s_{1}) = 2 q (a\cdot s_{1} \oplus 0) = q (a \cdot (2s_{1}) \oplus 0) = a \diamond q(2s_{1} \oplus 0) = a \diamond i_{1}(s_{1}). \]
We are now ready to complete the proof. Let $ \Tt$ be an operator $ \Aa$-system and $ \phi_{j} : \Ss_{j} \rightarrow \Tt $, $ j=1, \;2$ be two ucp $ \Aa$-bimodule maps. The map 
\begin{align*}
    \psi : \Ss \oplus \Tt& \rightarrow \Tt \\
    s_{1} \oplus s_{2} &  \mapsto \frac{1}{2} (\phi_{1}(s_{1}) + \phi_{2}(s_{2})),
    \end{align*} 
is unital, completely positive and $ \psi(a \oplus -a ) = 0$. This means that $ \mathcal{J} \subseteq \ker \psi$ and thus, by  \cite[Proposition 3.6]{Kavruk2010QuotientsEA}, there exists a unique unital completely positive map $ \Phi :  \Ss \oplus \Tt/ \mathcal{J}  \rightarrow \Tt $ such that $ \Phi \circ q = \psi$. Also, $ \Phi \circ i_{j} = \phi_{j} $, for $ j =1, 2$ and finally, 
\begin{align*}
    \Phi( a \diamond q (s_{1} \oplus s_{2})) &= \Phi(q (a\cdot s_{1} \oplus a \cdot s_{2} ))\\
    & = \frac{1}{2} (\phi_{1}(a \cdot s_{1}) + \phi_{2}(a \cdot s_{2}))\\
    & = \frac{1}{2} (a \cdot \phi_{1}(s_{1}) + a \cdot \phi_{2}(s_{2})) \\
    & = a \cdot \Big(\frac{1}{2} (\phi_{1}(s_{1}) + \phi_{2}(s_{2})) \Big) \\
    & = a \cdot \Phi(q (s_{1} \oplus s_{2}))
\end{align*}
for every $ a \in \Aa$, $ s_{1}\in \Ss$ and $ s_{2} \in \Tt$, i.e., $ \Phi$ is an $\Aa $-bimodule map. So $\Ss_{1} \oplus \Ss_{2}/ \mathcal{J}$ is an operator $ \Aa$-system that satisfies the universal property for coproducts  and thus it is completely order isomorphic to $\Ss_{1} \oplus_{\Aa} \Ss_{2} $ by uniqueness.

\end{proof}
\end{theorem}

So, for two faithful operator $ \Aa$-systems $ \Ss$ and $ \Tt$ we can always form their coproduct as a quotient operator system by a kernel $\Jj$. Moreover, an interesting situation occurs when the $ \Cs$-algebra $\Aa$ is also a von Neumann algebra. As we will see, in this case $ \Jj$ is a completely order proximinal kernel, i.e. we don't need to enlarge cones in order for the unit to be Archimedean.

\begin{proposition} \label{corollaryvN}
Let $\Aa$ be a unital $ \Cs$-algebra, $\Ss$, $\Tt$ be two operator $\Aa$-systems and $ \Hh$ be a Hilbert space such that $1\in  \Aa \subseteq \Ss \cap \Tt \subseteq \Bh$. If also $ \Aa$ is WOT-closed in $ \Bh$  then, the subspace 
\[    \mathcal{J} = \{ a \oplus -a : a \in \Aa \} \]
is a completely order proximinal kernel in  $ \Ss \oplus \Tt $ so that  $(\Ss \oplus \Tt /\mathcal{J}, \{D_{n} \}_{n\in \N}, 1\oplus1 + \Jj)$ is an operator system.

\begin{proof}
By Theorem \ref{theorimakernel}, we know that  $ (\Ss \oplus \Tt /\Jj, (C_{n})_{n\in \N}, 1\oplus 1 + \Jj )$ is an operator system. We will prove here that when $ \Aa$ is WOT-closed, then $ \Jj$ is a completely order proximinal kernel, i.e. $ C_{n}(\Ss \oplus \Tt /\Jj) = D_{n}(\Ss \oplus \Tt /\Jj)$, for all $ n\in \N$. It suffices to show that $ C_{1}(\Ss \oplus \Tt /\Jj) \subseteq D_{1}(\Ss \oplus \Tt /\Jj) $, equivalently,   that $ 1\oplus1 + \mathcal{J}$ is an Archimedean  order unit. 
Then, $ \Jj$ will  be a completely order proximinal kernel. Indeed, identify 
  \[ \Mn(\Ss\oplus \Tt/\Jj) = \Mn(\Ss \oplus \Tt)/ \Mn(\Jj) \]
  and note that $ \Mn(\Jj)$ is also a WOT-closed selfadjoint subspace  with no other positive elements than 0. Furthermore, $ \Mn(\Jj) = \{ [a_{i,j}]\oplus [-a_{i,j}]: [a_{i,j}] \in \Mn(\Aa) \}$, via the canonical shuffle, so by repeating the proof at each matrix level we will be done. 

Let $ s \oplus t + \mathcal{J} \in \Ss \oplus \Tt /\mathcal{J} $ be such that 
\begin{align*}
    \e (1 \oplus 1 + \mathcal{J}) + s \oplus t + \mathcal{J} \in D_{1}, \; \; \forall \e >0, 
\end{align*}
where $ s=s^{*}$ and $ t = t^{*}$. Equivalently,
$  \e (1 \oplus 1) + s \oplus t + \mathcal{J} \in D_{1}$, for every $ \e >0$. So, for every $ \e >0$ there exists an $ a_{\e} \in \Aa_{h}$ such that 
\[ \e (1 \oplus 1) + s \oplus t + a_{\e} \oplus -a_{\e} \geq 0.  \]
Define 
\begin{align*}
    X_{n} = \{ a \in \Aa_{h} : \frac{1}{n} (1 \oplus 1) + s \oplus t + a \oplus -a \geq 0 \} 
\end{align*}
and note that $ X_{m} \subseteq X_{n}$ for every two integers $ 1\leq n\leq m$. 

\textit{Claim}: $ \{ X_{n} \}_{n \geq 1}$ is a  decreasing sequence of non empty, WOT-compact sets.\\
Clearly the sets $ X_{n}$ are  non-empty for every $ n \in \N$. We will show each $X_{n}$ is WOT-closed and norm-bounded, thus WOT-compact. Let $ n \geq 1$ and  $ (a_{i})_{i\in I}$ be a net in $ X_{n}$ such that $ a_{i} \xrightarrow{WOT} a $. Of course, since $\Aa$  is WOT-closed, we have $ a \in \Aa$. Also, 
\begin{align*}
    \sca{\big(a\oplus -a + \frac{1}{n} \cdot 1\oplus 1 + s\oplus t\big)(u), u}= \lim_{i} \sca{\big(a_{i}\oplus -a_{i} + \frac{1}{n} \cdot 1\oplus 1 + s\oplus t\big)(u), u}\geq 0
\end{align*}
for every $ u  \in \Hh \oplus \Hh$. This means that $ a \in X_{n}$ and so $ X_{n}$ is WOT-closed. \\
Let $ n =1$, we will show that $ X_{1}$ is $\normdot$-bounded, which implies that all the $X_{n}$ are bounded. Let $ a \in X_{1}$, that is, 
\[ 1\oplus 1 + s\oplus t + a \oplus -a \geq 0 \]
i.e.

\begin{align*}
    1 + s + a \geq0 \; \; \text{ and } \; \; 1 +t -a \geq 0.
\end{align*}
So this means that 
\begin{align*}
        -(1+s) \leq a \leq 1+t,
\end{align*}
and since $ a \in \Aa$ is hermitian we obtain that 
\begin{align*}
    \norm{a} \leq \max\{ \norm{s+1}, \norm{t+1}\}.
\end{align*}
 We deduce that $ \sup_{a\in X_{1}}\norm{a} \leq  \max\{ \norm{s+1}, \norm{t+1}\} $, that is, $X_{1}$ is bounded. Finally, we have that $ \{X_{n} \}_{ n \geq 1 }$ is a  decreasing sequence of non empty WOT-compact sets and consequently it has non-empty intersection, i.e., there exists an $ a_{0} \in \bigcap_{n\geq 1}X_{n}$. So, for this $ a_{0} \in \Aa$, 
\begin{align*}
    \frac{1}{n} \cdot 1\oplus 1 + s\oplus t + a_{0}\oplus-a_{0} \geq 0 \; \; \; \text{ for all } n >0,
\end{align*}
but this means that $s\oplus t + a_{0}\oplus-a_{0} \geq 0$, equivalently, $ s\oplus t + \mathcal{J} \in D_{1}$.
\end{proof}
\end{proposition}

Note that in the above result, the operator $\Aa$-systems are automatically faithful operator $\Aa$-systems.
Now we state an algebraic result for the coproduct, that highlights the amalgamation that takes place inside the coproduct.

\begin{corollary} \label{intersection}
Let $\Aa$ be a unital $ \Cs$-algebra and $\Ss$, $\Tt$ be two faithful operator $\Aa$-systems. Let also $ \Ss \oplus_{\Aa} \Tt$ be their coproduct over $\Aa$ and $i_{1}$, $ i_{2}$ the associated complete order embeddings, from $\Ss$, $\Tt$, respectively, into the coproduct. 
 Then  
\begin{align*}
    i_{1}(\Ss) \cap i_{2}(\Tt) = \Aa,
\end{align*}
completely order isomorphically.
    Thus, we may omit the inclusion maps and assume that $\Ss$, $\Tt$ and $\Aa$ are subsystems of $ \Ss \oplus_{\Aa} \Tt$ such that $ \Ss\cap \Tt =\Aa$.
    \begin{proof}
   By Remark \ref{sameembedding}, assume that $\Aa \subseteq \Ss \cap \Tt \subseteq \Bh$, where $ \Hh$ is a Hilbert space and write the coproduct as in Theorem \ref{theorimakernel}
   \[ \Ss \oplus_{\Aa} \Tt = \Ss \oplus \Tt/\Jj.  \]
   Now let $$ u= i_{1}(s) = i_{2}(t) \in  i_{1}(\Ss) \cap i_{2}(\Tt), $$ i.e.  $ u= 0\oplus 2t +\Jj= 2s\oplus 0 + \Jj$. This implies that $ 2( s \oplus -t) \in \Jj$, that is, $a:= s=t  \in \Aa$ and thus  $$ u = i_{1}(a)= 2a\oplus 0 +\Jj = a\oplus a +\Jj = 0 \oplus 2a  + \Jj = i_{2}(a).$$ So, if we define 
   \begin{align*}
        j: \Aa &\rightarrow \Ss\oplus \Tt /\Jj \\
             a & \mapsto a \oplus a +\Jj
   \end{align*}
   we immediately see that $ j(a) = i_{1}(a)= i_{2}(a)$ for all $ a \in \Aa$ and that $j$ is a complete order embedding.
    \end{proof}
\end{corollary}

\subsection{REMARKS AND EXAMPLES}

\begin{definition}[Conditional expectation] \label{conditionalexpectation}
Let $ 1 \in \Tt \subseteq \Ss $ be two operator systems. A conditional expectation $ \phi : \Ss \rightarrow \Tt$, is a u.c.p. map such that $\phi(t)= t$, for all $ t \in \Tt$.
\end{definition}

\begin{remark}
 Note that, when $ \Tt=\Aa$ is a $\Cs$-algebra then by Proposition \ref{bimoduleequivalence}, a conditional expectation $ \phi : \Ss \rightarrow \Aa$ is a u.c.p. $\Aa$-bimodule map. This is also equivalent to being a u.c.p. projection onto $ \Aa$. 
\end{remark}

\begin{remark}
Corollary \ref{intersection} gives the following convenient form for the coproduct of two faithful operator $\Aa$-systems $\Ss$ and $\Tt$. Suppose that there exist two conditional expectations $ E_{S}: \Ss \rightarrow \Aa$ and $ E_{T}: \Tt \rightarrow \Aa$ and define $\Ss_{0}:=\ker E_{S} $ and $ \Tt_{0}:= \ker E_{T}$. Then, as linear spaces
\[ \Ss \oplus_{\Aa} \Tt \cong \Aa \oplus \; \Ss_{0} \oplus  \Tt_{0},\]
where the sums are direct.
   
Indeed, decompose as $\Ss = \Aa \oplus \Ss_{0} $, $ \Tt= \Aa \oplus \Tt_{0}$ and so, by Corollary \ref{intersection} we may assume that $\Ss$ and $\Tt$ are subsystems of $\Ss\oplus_{\Aa} \Tt$ and $ \Ss \cap \Tt = \Aa$, i.e. $ \Ss_{0} \cap \Tt_{0} =\{0\}.$
\end{remark}
We also have another way to interpret the $\Aa$-coproduct algebraically when there exist conditional expectations onto $\Aa$.
\begin{proposition}
Let $ \Ss$, $\Tt$ be two faithful operator $\Aa$-systems and $ E_{\Ss}$, $ E_{\Tt}$ two conditional expectations from $\Ss$ and $\Tt$ respectively onto $\Aa$. Let $\Rr := \{ s \oplus t \in \Ss \oplus \Tt: E_{\Ss}(s) = E_{\Tt}(t) \}$, be an operator subsystem of $ \Ss \oplus \Tt$.
Then as linear spaces 
\[ \Rr \cong \Ss \oplus_{\Aa} \Tt \]
via the restriction of the quotient map to $\Rr$.
However, this map need not always be a complete order isomorphism, since in fact its inverse may even fail to be positive. \end{proposition}
\begin{proof}
Let $ q : \Ss \oplus \Tt \rightarrow \Ss \oplus_{\Aa} \Tt $ denote the quotient map. 
 Then, its restriction  
\[ q\arrowvert_{\Rr}: \Rr \rightarrow \Ss \oplus_{\Aa} \Tt ,\] 
is a unital bijection and a completely positive map (since the quotient map is always unital and completely positive).

But, there exist operator $ \Aa$-systems, for which the map $q\arrowvert_{\Rr} $ does not have a positive inverse. Indeed,  consider $ \Ss = \Tt = M_{2}$, as operator systems that are bimodules over $ \mathcal{D}_{2}$, the algebra of diagonal 2 by 2 matrices. Their $\mathcal{D}_{2}$-coproduct exists and there exist  conditional expectations $ E : M_{2} \rightarrow \mathcal{D}_{2}$, namely,  the projection to the diagonal, their $\mathcal{D}_{2}$-coproduct is represented as $$ M_{2}\oplus_{\mathcal{D}_{2}} M_{2} = M_{2}\oplus M_{2}/ \{ D \oplus -D: D\in \mathcal{D}_{2} \}.$$  Let 
\begin{align*}
    s= \begin{bmatrix} 2 & \frac{5}{2}\\
    \frac{5}{2} & 2
    \end{bmatrix} \; \; \text{ and } \; \;  t= \begin{bmatrix} 2& 0\\
    0 & 2
    \end{bmatrix},
\end{align*}
and note that $ s \oplus t \in \Rr$, since they have the same diagonal.
Then, \begin{align*}
        \begin{bmatrix} 2 & \frac{5}{2}\\
    \frac{5}{2} & 2
    \end{bmatrix} \oplus  \begin{bmatrix} 2& 0\\
    0 &2  \end{bmatrix} + \mathcal{J} \geq 0,
    \end{align*}
    since if we chose $ I_{2} \in \mathcal{D}_{2}$, then \begin{align*}
    s+ I_{2} = \begin{bmatrix}
        3 & \frac{5}{2}\\
    \frac{5}{2} & 3
    \end{bmatrix} \geq 0 \; \; \text{ and } \; \; t - I_{2} = \begin{bmatrix}
        1&0 \\
        0&1
    \end{bmatrix} \geq 0,
    \end{align*}
but \begin{align*}
        \begin{bmatrix} 2 & \frac{5}{2}\\
    \frac{5}{2} & 2
    \end{bmatrix} \oplus  \begin{bmatrix} 2& 0\\
    0 &2  \end{bmatrix} \ngeq 0
    \end{align*} 
     since the matrix $  \begin{bmatrix} 2 & \frac{5}{2}\\
    \frac{5}{2} & 2
    \end{bmatrix}$ isn't positive.

\end{proof}

By the above remarks, we have the following.
\begin{corollary}
Suppose that $\Ss$ and $\Tt$ are two finite dimensional faithful operator $\Aa$-systems, that are bimodules over a finite dimensional $\Cs$-algebra $ \Aa$.  Then,  $ \dim(\Ss\oplus_{\Aa}\Tt)= \dim(\Ss) + \dim(\Tt) - \dim(\Aa)$. 
\end{corollary}

\hfill

Graph operator systems are a nice example of (faithful) operator $\Aa$-systems.

\begin{definition}
Let $G=(V,E)$ be a graph on $n$ vertices. We may define an operator system $\Ss_{G} \subseteq \Mn$ as follows, 
\[ \Ss_{G}= \{ T \in \Mn : T_{i,j} \neq 0 \Rightarrow
i=j \; \text{ or } \; (i,j) \in E \}.   \]
We will call such an operator system $\Ss_{G} $, the  graph operator system of $ G$.
\end{definition}

Then $ \Ss_{G}$ is a $\Dd_{n}$-bimodule and conversely, if $ \Ss \subseteq \Mn$ is an operator system that is a $\Dd_{n}$-bimodule then there exists a graph such that $ \Ss= \Ss_{G}$. Indeed, let $ V= [n]$ and define  
$ E=\{(i,j) : i \neq j \; \text{ and } \; E_{i,i}\mathcal{S}E_{j,j} \neq \{0\}\}$. Note that we can write $\Ss_{G}$ as $\Ss_{G}= \spann \{E_{i,j}: i=j \; \text{ or } \;  (i,j) \in E \}$.

\begin{example}
Let $G_{k}= (V_{k},E_{k})$, $ k= 1,2 $ be two graphs on $ n$ vertices. Consider their graph operator systems
\[ \Ss_{G_{k}}= \spann \{E_{i,j}: i=j \; \text{ or } \;  (i,j) \in E_{k} \} \subseteq M_{n}, \; \; \; \;  k = 1,2, \]
where $E_{i,j}  $ are the matrix units of $M_{n}$. We know that the operator systems $ \Ss_{k}$ are $ \mathcal{D}_{n}$-bimodules, where $ \mathcal{D}_{n}$ is the $\Cs$-algebra of diagonal matrices. So they are operator $ \mathcal{D}_{n}$-systems.  Now, define 
\begin{align*}
    \mathcal{J} = \{ D \oplus -D : D \in  \mathcal{D}_{n}\}
\end{align*}
and note that this is a kernel. Indeed, by proposition \ref{corollaryvN} ($\Dd_{n}$ is a von Neumann algebra) $ \Jj$ is in fact a completely proximinal kernel. So, we have that 
\begin{align*}
    \Ss_{G_{1}} \oplus_{\mathcal{D}_{n}} \Ss_{G_{2}} \cong_{c.o.i.} \Ss_{G_{1}} \oplus \Ss_{G_{2}}/ \mathcal{J},  
\end{align*}
where $\Ss_{G_{1}} \oplus_{\mathcal{D}_{n}} \Ss_{G_{2}}$ is the operator $ \mathcal{D}_{n}$-system coproduct of $ \Ss_{G_{1}}$ and $  \Ss_{G_{2}}$.

\end{example}
 
 So, the coproduct of two graph operator systems is an operator $ \Dd_{n}$-system. But is it a graph operator system?
We will answer this question in Section \ref{graphs}, after establishing first some more theory.

\begin{remark}
In fact, we can extend the graph operator system case to the infinite dimensional case. Indeed, let $(X,\mu)$ be a $\sigma$-finite measure space.  Since $  L^{\infty}(X,\mu) $ can be identified as the $\Cs$-algebra of multiplication operators in $\B(L^{2}(X,\mu))$ we have that $\B(L^{2}(X,\mu))$ is a bimodule over $  L^{\infty}(X,\mu) $. Since  $  L^{\infty}(X,\mu) $ is  a von Neumann algebra, by Proposition \ref{corollaryvN},  if we consider two completely order isomorphic copies of the operator system $ \B(L^{2}(X,\mu))$ we may write their operator $L^{\infty}(X,\mu)$-system coproduct   as
 \[ \B(L^{2}) \oplus_{L^{\infty}} \B(L^{2}) \cong_{c.o.i.} \B(L^{2}) \oplus \B(L^{2})/ \{M_{f} \oplus - M_{f}: f \in L^{\infty}   \}. \]
We will not deal with this case in this work, we intend, however, to study the matter in a subsequent work.
\end{remark}

\section{C*-COVERS OF OPERATOR A-SYSTEMS} \label{sectioncovers}
Let $ \Ss$ be an operator system. A $\Cs$-cover of $\Ss$ is a pair $ ( \Cc, i)$, where $ \Cc$ is a unital $\Cs$-algebra and $ i: \Ss \rightarrow \Cc$ is a unital completely isometric map such that $ i(\Ss)$ generates $ \Cc$ as a $\Cs$-algebra.

\subsection{A UNIVERSAL C*-COVER}

For an operator system $\Ss$, its universal $\Cs$-algebra \cite{KIRCHBERG1998324} is defined as the unique $ \Cs$-algebra $C_{u}^{*}(\Ss)$ generated by $\Ss$ such that for any other $\Cs$-algebra $\B$ and u.c.p. map $ \phi :\Ss \rightarrow \B$, there exists a *-homomorphism $ \pi_{\phi} :C_{u}^{*}(\Ss) \rightarrow \B $ that extends $\phi$. We will  extend this notion to the category of operator $ \Aa$-systems. The  difference is that now, our unique $\Cs$-algebra will be an $ \Aa$-bimodule, and the u.c.p. maps and *-homomorphisms, will be $\Aa$-bimodule maps.

Let $ \Ss$ be an operator system that is a bimodule over a unital $\Cs$-algebra $\Aa$. We begin by constructing the free *-algebra $\Aa$-bimodule:
\begin{align*}
    \mathcal{F}(\Ss):= \Ss \; \oplus \; ( \Ss \otimes_{\Aa}\Ss) \; \oplus \; ( \Ss \otimes_{\Aa}\Ss\otimes_{\Aa}\Ss) \; \oplus \;  \cdots,
\end{align*} where $\Ss^{\otimes (n)}=  \Ss \otimes_{\Aa} \cdots \otimes_{\Aa}\Ss $ is the algebraic tensor product of $\Aa$ bimodules (see for example \cite[section 1.3]{Hazewinkel}). The space  $\mathcal{F}(\Ss)$  becomes a *-algebra if we define multiplication by: 
\begin{align*}
   ( s_{1}\otimes \cdots \otimes s_{n}, s'_{1} \otimes \cdots \otimes s'_{m}) \mapsto s_{1}\otimes \cdots \otimes s_{n}\otimes s'_{1} \otimes \cdots \otimes s'_{m},
\end{align*}
 the *-operation by :
\begin{align*}
    (s_{1} \otimes \cdots \otimes  s_{n})^{*}= s_{n}^{*} \otimes \cdots \otimes s_{1}^{*}
\end{align*}
and module action by:
\begin{align*}
    a \cdot ( s_{1}\otimes \cdots \otimes s_{n})= (a \cdot s_{1}) \otimes \cdots \otimes s_{n} \\
      ( s_{1}\otimes \cdots \otimes s_{n}) \cdot a =   s_{1} \otimes \cdots \otimes (s_{n} \cdot a) 
\end{align*}
and extend all linearly to the tensor product and then to the algebraic direct sum. 
Now, let  $ \phi : \Ss \rightarrow \B(\Hh) $ be a unital completely positive map that is also an $\Aa$-bimodule map in the sense that there exists a *-homomorphism $ \rho: \Aa \rightarrow \Bh$ such that for all $ s\in \Ss$ and $ a \in \Aa$,
\[ \phi(a \cdot s) = \rho(a) \phi(s). \]
Note that by Theorem \ref{asystemsrepresentation} such a map always exists.
Any such map
gives rise to a  *-homomorphism $\pi_{\phi} : \mathcal{F}(\Ss) \rightarrow \B(\Hh)$ by setting 
\begin{align*}
    \pi_{\phi}( s_{1}\otimes \cdots \otimes s_{n}) = \phi(s_{1}) \cdots \phi(s_{n}),
\end{align*}
and extending linearly to the tensor product and then to the direct sum. Observe that $\pi_\phi$ is an $\Aa$-module map. Now let $ u \in \mathcal{F}(\Ss) $ and define $\norm{u}_{\mathcal{F}(\Ss)} := \sup_{\phi}\norm{\pi_{\phi}(u)}$, where the supremum is taken over all such unital completely positive $ \Aa$-bimodule maps $\phi$.

Finally, define  $C^{*}_{u,\Aa}(\Ss) $ to be the completion of the quotient of $\mathcal{F}(\Ss) $ by the ideal 
$$\mathcal{N}=\{ u \in \mathcal{F}(\Ss) : \norm{u}_{\mathcal{F}(\Ss) }=0 \}.$$

By construction  $C^{*}_{u,\Aa}(\Ss)  $, possesses a universal property which can be stated as follows:
If $ \B$ is a unital $ \Cs$-algebra that is also a bimodule over $ \Aa$ and $ \phi : \Ss \rightarrow \B$ is a u.c.p. $ \Aa$-bimodule map, then there exists an $\Aa$-bimodule *-homomorphism  $\pi :C^{*}_{u,\Aa}(\Ss)  \rightarrow \B $ such that $ \pi\arrowvert_{\Ss} = \phi$. It is not hard to see that $C^{*}_{u,\Aa}(\Ss)$ is unique with respect to this universal property.

Sometimes we may identify the $ \Ss$ with its image inside $C^{*}_{u,\Aa}(\Ss) $, and consider it as an operator subsystem. 
\begin{proposition} \label{suminuniversals}
Let $\Ss_{1}$ and $ \Ss_{2}$ be two faithful operator $\Aa$-systems. Then, their operator $\Aa$-system coproduct $    \Ss_{1}\oplus_{\Aa} \Ss_{2}$ is completely order isomorphic to the operator subsystem
\begin{align*}
     \Ss_{1} +\Ss_{2}:= \{s_{1}+s_{2}: s_{1}\in \Ss_{1}, \; s_{2} \in \Ss_{2} \} \subseteq C^{*}_{u,\Aa}(\Ss_{1})*_{\Aa}C^{*}_{u,\Aa}(\Ss_{2}),  
\end{align*}
where $C^{*}_{u,\Aa}(\Ss_{1})*_{\Aa}C^{*}_{u,\Aa}(\Ss_{2})$ is the $\Cs$-algebra free product amalgamated over $\Aa$.
\begin{proof}
   Let $\Tt\subseteq \B(\Kk)$ be an operator $ \Aa$-system and let also
   \begin{align*}
       \phi_{1}:\Ss_{1} \rightarrow \Tt \subseteq \B(\Kk)\\
       \phi_{2}:\Ss_{2} \rightarrow \Tt \subseteq \B(\Kk)
   \end{align*}
   be two u.c.p. $\Aa$-bimodule maps. By the universal property of the universal $\Aa$-bimodule $\Cs$-algebras, there exist two $\Aa$-bimodule *-homomorphisms 
  
   \begin{align*}
       \pi_{\phi_1} : C^*_{u,\Aa}(\Ss_1) \rightarrow \B(\Kk)\qquad \text{and }\qquad
       \pi_{\phi_2} : C^{*}_{u,\Aa}(\Ss_2) \rightarrow \B(\Kk),
   \end{align*}
   that extend $ \phi_{1}$ and $ \phi_{2}$ respectively and restrict to a common *-representation of $\Aa$. Now, from the universal property of the amalgamated free product, there exists a *-homomorphism 
   $$ \pi: C^{*}_{u,\Aa}(\Ss_{1})*_{\Aa}C^{*}_{u,\Aa}(\Ss_{2}) \rightarrow \B(\Kk) $$ 
   such that $ \pi\arrowvert_{C^{*}_{u,\Aa}(\Ss_{i})} = \pi_{\phi_{i}}$, for both $ i =1, 2$. Finally, set 
   \begin{align*}
       \Phi:=\pi\arrowvert_{\Ss_{1}+\Ss_{2}}: \Ss_{1}+\Ss_{2} \rightarrow \B(\Kk),
   \end{align*}
   and note that this is a unital completely positive $\Aa$-bimodule map such that $\Phi\arrowvert_{\Ss_{i}}= \phi_{i}$, for $ i=1, 2$ and whose image lies inside $\Tt$.
   \end{proof}
\end{proposition}

\begin{proposition}
Let $ \Ss_{1}$ and $\Ss_{2}$ be two faithful operator $\Aa$-systems.  Then,
\begin{align*}
  C^{*}_{u,\Aa}(\Ss_{1}\oplus_{\Aa} \Ss_{2}) \cong   C^{*}_{u,\Aa}(\Ss_{1})*_{\Aa}C^{*}_{u,\Aa}(\Ss_{2}).
\end{align*}
\begin{proof}
  We will prove that $C^{*}_{u,\Aa}(\Ss_{1})*_{\Aa}C^{*}_{u,\Aa}(\Ss_{2})$ possesses the universal property  for the universal $\Cs$-algebra $C^{*}_{u,\Aa}(\Ss_{1}\oplus_{\Aa} \Ss_{2})$; the result will follow by uniqueness.  Using  Proposition \ref{suminuniversals}, we identify $\Ss_{1}\oplus_{\Aa} \Ss_{2} $ with the operator subsystem $ \Ss_{1}+\Ss_{2}$ of the amalgamated $\Cs$-algebra free product  $  C^{*}_{u,\Aa}(\Ss_{1})*_{\Aa}C^{*}_{u,\Aa}(\Ss_{2})$ and denote by 
   \begin{align*}
       i_{1}:\Ss_{1} \rightarrow \Ss_{1}\oplus_{\Aa} \Ss_{2} \subseteq   C^{*}_{u,\Aa}(\Ss_{1})*_{\Aa}C^{*}_{u,\Aa}(\Ss_{2})\\
       i_{2}:\Ss_{2} \rightarrow \Ss_{1}\oplus_{\Aa} \Ss_{2} \subseteq   C^{*}_{u,\Aa}(\Ss_{1})*_{\Aa}C^{*}_{u,\Aa}(\Ss_{2}) 
   \end{align*}
   the inclusion maps of $\Ss_{1}$ and $\Ss_{2}$ respectively. 
   
Let $\B$ be a unital $\Cs$-algebra that is also an $\Aa$-bimodule. Let also
\begin{align*}
    \Phi : \Ss_{1}\oplus_{\Aa} \Ss_{2} \rightarrow \B
\end{align*}
be a u.c.p. $\Aa$-bimodule map. Then,

\begin{align*}
    \Phi \circ i_{k} :\Ss_{k} \rightarrow \B, \; \; \; k=1,2
\end{align*}
   are also  u.c.p. $\Aa$-bimodule maps.
   By the universal property of the universal $\Cs$-algebras $C^{*}_{u,\Aa}(\Ss_{i}) $, there exist *-homomorphisms 
   \begin{align*}
        \rho_{1}: C^{*}_{u,\Aa}(\Ss_{1}) \rightarrow   \B\\
      \rho_{2}: C^{*}_{u,\Aa}(\Ss_{2}) \rightarrow   \B
   \end{align*}
   that extend $\Phi \circ i_{1}$ and $ \Phi \circ i_{2}$ respectively. Also $ \rho_{1}$ and $ \rho_{2}$ agree on $\Aa$, so by the universal property of amalgamated free products, there exists a *-homomorphism
   \begin{align*}
       \rho : C^{*}_{u,\Aa}(\Ss_{1})*_{\Aa}C^{*}_{u,\Aa}(\Ss_{2}) \rightarrow \B
   \end{align*}
such that $ \rho\arrowvert_{C^{*}_{u,\Aa}(\Ss_{k})}= \rho_{k} $, $ k=1,2$.   
   
Finally, for each $k=1,2$,
\begin{align*}
    \rho\arrowvert_{\Ss_{k}}= \rho_{k}\arrowvert_{\Ss_{k}}= \Phi \circ i_{k},
\end{align*}
which implies that $ \rho\arrowvert_{\Ss_{1}\oplus_{\Aa} \Ss_{2}} = \Phi$, i.e., $\rho$ extends $\Phi$ and thus the proof is complete.

\end{proof}
\end{proposition}

\subsection{THE C*-ENVELOPE}

We can now deal with the other famous $\Cs$-cover of operator systems, that is, the $\Cs$-envelope of an operator system. 

The $\Cs$-envelope $ C^{*}_{e}(\Ss)$ of an operator system $ \Ss$ is defined to be the $ \Cs$-algebra generated by $ \Ss$ in its injective envelope $ I(\Ss)$ \cite{1979773}. The $\Cs$-envelope $ C^{*}_{e}(\Ss)$ is the unique $\Cs$-cover having  the following universal property: For any $\Cs$-cover $ i : \Ss \hookrightarrow \Aa$ there exists a unique unital *-homomorphism  $\pi : \Aa \rightarrow C^{*}_{e}(\Ss)$ such that $ \pi(i(s)) =s$ for every $ s \in \Ss$.

In \cite{Duncan2007CenvelopesOU}, Duncan proved that under certain assumptions, the $\Cs$-envelope of the amalgamated free product of two operator algebras $\Aa_1$ and $\Aa_2$ is *-isomorphic to the amalgamated free product of their respective $\Cs$-envelopes. See also \cite[Section 5.3]{Davidson2014SemicrossedPO}. It was assumed that  for every *-representation $ \pi $ of their $\Cs$-envelope $ C^{*}_{e}(\Aa_{i})$, the restriction $\pi\arrowvert_{\Aa_{i}}$ to $\Aa_{i}$ has the unique extension property. This is called  \textit{hyperrigidity} by Arveson  \cite{Arveson2008TheNC}.

Here we prove the analogous result  for operator systems, i.e. we assume the operator systems to be hyperrigid and prove that the $C^{*}$-envelope of the coproduct of two operator systems is *-isomorphic to the  free product of their $\Cs$-envelopes, amalgamated over the unit. In fact we will prove the result for operator $\Aa$-systems (Theorem \ref{C*envelopesofAsystems}) and the operator system case  will follow as a special case.

In \cite{Kavruk2014}, Proposition 5.6, the author showed that if an operator system $ \Ss \subseteq \Aa$ contains enough unitaries to generate $\Aa$, then $ \Aa \cong C^{*}_{e}(\Ss)$. This was used in the following result.
\begin{theorem}\cite[Theorem 5.2]{bbf61d43a5de485bad4fe7b6c908f9a2} \label{unitaries}
Let $ \Ss_{1} \subseteq \Aa_{1}$ and  $ \Ss_{2} \subseteq \Aa_{2}$ be two operator systems,  where $ \Aa_{1}, \Aa_{2}$ are unital $\Cs$-algebras. Then, if  the operator systems $ \Ss_{i}$, $ i=1,2$ contain enough unitaries to generate  $ \Aa_{i}$ as $\Cs$-algebras, then $ \Ss_{1}\oplus_{1} \Ss_{2}$ contain enough unitaries to generate $\Aa_{1} *_{1} \Aa_{2}$ as a $\Cs$-algebra, and 
\[ C^{*}_{e}( \Ss_{1}\oplus_{1} \Ss_{2})  \cong \Aa_{1} *_{1} \Aa_{2}. \]
\end{theorem}

The above result provides a sufficient condition for the $\Cs$-envelope of the coproduct of two operator systems to be *-isomorphic with the amalgamated free product of their $\Cs$-envelopes, that is, to assume that the operator systems contain the unitaries that generate their $\Cs$-envelopes. We replace this condition with hyperrigidity for operator systems, therefore strengthening the above result.

\begin{definition}\cite[Unique extension property for u.c.p. maps]{Arveson2003NOTESOT}
Let $ \Ss$ be an operator system and $ (\Aa, i )$ be a $ \Cs$-cover. A u.c.p. map $ \phi: \Ss \rightarrow \Bh$, is said to have the unique extension property, if it has a unique extension to a completely positive map  $ \Tilde{\phi}: \Aa \rightarrow \Bh$ that is also a *-representation. 
\end{definition}

\begin{definition}\cite[Hyperrigidity]{Arveson2008TheNC}
Let $\Ss \subseteq \Aa $ be an operator system and $ (\Aa, i )$ be a $ \Cs$-cover. The operator system $\Ss$ is called hyperrigid in $ \Aa$, if for every representation $\pi: \Aa \rightarrow \Bh$, its restriction $\pi \arrowvert_{\Ss}$  has the unique extension property. 
\end{definition}

\begin{definition}\cite{Kavruk2010QuotientsEA}
Let $\Ss\subseteq \Aa$ be an operator subsystem of the unital $\Cs$-algebra $ \Aa$. We say that $ \Ss$ contains enough unitaries in $\Aa$, if the unitaries in $\Ss$ generate $\Aa$ as a $\Cs$-algebra.
\end{definition}

\begin{proposition} \cite[Lemma 9.3]{Kavruk2010QuotientsEA} \label{homoextension}
Let $\Ss \subseteq \Aa$ be an operator system that contains enough unitaries in  $ \Aa$. Then, every u.c.p. map $ \phi : \Ss \rightarrow \Bh$ that preserves unitaries, extends uniquely to a unital completely positive map $ \Tilde{\phi} : \Aa \rightarrow \Bh$, which is also a *-homomorphism. That is, any such u.c.p. map $ \phi : \Ss \rightarrow \Bh$ has the unique extension property.
\end{proposition}

The two following results can be found in \cite{HARRIS20192156}. We include them with somewhat different proofs.

\begin{proposition}  \label{hyperrigidimpliesunitaries}
Let $\Ss \subseteq \Aa$ be an operator system that contains enough unitaries in its $\Cs$-cover $\Aa$, then $\Ss$ is hyperrigid in $\Aa$. 
\begin{proof}
Let $ \pi : \Aa \rightarrow \Bh$ be a representation of $\Aa$, and $ \phi :\Aa \rightarrow \Bh$, be a u.c.p. map such that $ \phi\arrowvert_{\Ss}= \pi\arrowvert_{\Ss}$. Then, for every unitary $ u \in \Ss$, we have that $ \phi(u)= \pi(u)$, which means that $ \phi(u)$ is also unitary ($\pi$ is a *-homomorphism). So, by Proposition \ref{homoextension}, $ \phi\arrowvert_{\Ss}$ extends uniquely to a u.c.p. map on $\Aa$ that is also a *-homomorphism and since $\pi$ is another such extension, $ \phi = \pi$ on all of $\Aa$.
\end{proof}
\end{proposition}

The following proposition implies that if an operator system is hyperrigid in one of its $\Cs$-covers, it is necessarily hyperrigid in its $ \Cs$-envelope. In fact:
\begin{proposition}\label{allcstarcovers}
Let $\Ss$ be an operator system and $ (\Aa,i)$ be a $\Cs$-cover of $\Ss$. If $\Ss$ is hyperrigid in $\Aa$, then $\Aa \cong C^{*}_{e}(\Ss)$ via a *-isomorphism. In fact, the  conclusion holds if we only assume the existence of a faithful representation $ \rho :\Aa \rightarrow \Bh$ such that $\rho\arrowvert_{\Ss} $ has the unique extension property.
\begin{proof}
First, identify $\Ss  $ with its image inside its $\Cs$-envelope, that is, assume $ \Ss \subseteq C^{*}_{e}(\Ss) $. Let $\pi : \Aa \rightarrow C^{*}_{e}(\Ss)$ be the associated surjective *-homomorphism such that $ \pi ( i(s) )= s$ for all $ s \in \Ss$. Let $\rho : \Aa \rightarrow \Bh$ be a faithful *-representation of the unital $\Cs$-algebra $\Aa$. By Arveson's extension theorem, there is a  u.c.p. map $ \phi : C^{*}_{e}(\Ss) \rightarrow \Bh$ that extends $ \rho \circ i : \Ss \rightarrow \Bh$. Now we see that $ \phi \circ \pi (i(s))= \phi(s)= \rho(i(s))$ for all $ s \in \Ss$. But since $ \Ss$ is hyperrigid in $\Aa$, $ \rho$ has the unique extension property and thus we must have that $ \phi \circ \pi = \rho$ in all of $ \Aa$. Now $ \phi \circ \pi $ is injective, since  $\rho$ is faithful, and thus $ \pi $ is also injective. Together with being surjective we conclude that $ \pi$ is a *-isomorphism.
\end{proof}
\end{proposition}

\begin{remark}
It is known that the injective envelope respects module actions, in the sense that if  $ \Ss$ is an operator $\Aa$-system, there always exist a unital complete order embedding $ \phi :\Ss \rightarrow I(\Ss)$ and a unital *-homomorphism $ \pi : \Aa \rightarrow I(\Ss) : a \mapsto a \cdot e$ into its injective envelope with $ \phi(a \cdot s ) = \pi(a) \phi(s)$ for all $ s \in \Ss$ and $ a \in \Aa$ \cite[Theorem 15.12]{paulsen_2003}. In fact, Theorem \ref{asystemsrepresentation} is an application of this result. Consequently, the $\Cs$-envelope of $\Ss$ respects module actions as well. Indeed, let $ C^{*}_{e}(\Ss)= C^{*}(\phi(\Ss)) \subseteq I(\Ss)$ and note that for all $ a \in \Aa$,
\[  \pi(a) = \phi(a \cdot e) \in \phi(\Ss) \subseteq C^{*}_{e}(\Ss).\]
Moreover, since $ \phi(\Ss)$ is a $ \pi(\Aa)$-bimodule in $ I(\Ss)$, then by continuity of the multiplication, so is  $C^{*}_{e}(\Ss) $. Finally, when we assume $ \Ss$  to be a faithful operator $\Aa$-system then by the above arguments there exist a unital c.o.e. $\phi: \Ss \rightarrow C^{*}_{e}(\Ss)$ and a unital injective *-homomorphism $\pi : \Aa \rightarrow C^{*}_{e}(\Ss)$ such that $ \phi(a \cdot s)= \pi(a) \phi(s)$ for all $ a \in \Aa$ and $ s \in \Ss$.
\end{remark}

Now we are ready to prove the main result of this section.

 \begin{theorem} \label{C*envelopesofAsystems}
Let $\Ss_{1} $ and $\Ss_{2}$ be two faithful operator $\Aa$-systems that are hyperrigid in their respective $\Cs$-envelopes, then   $$C^{*}_{e}(\Ss_{1}  \oplus_{\Aa} \Ss_{2}) \cong C^{*}_{e}(\Ss_{1}) *_{\Aa} C^{*}_{e}(\Ss_{2}).  $$
 \end{theorem}
\begin{proof}

Since $ \Ss_{i}$, $ i=1,2$ are faithful operator $ \Aa$-systems,  there exist unital *-embeddings $ \e_{i} : \Aa \rightarrow C_{e}^{*}(\Ss_{i})$. By definition of the amalgamated free product, there exist isometric *-homomorphisms $$\pi_{i}: C^{*}_{e}(\Ss_{i}) \rightarrow C^{*}_{e}(\Ss_{1}) *_{\Aa} C^{*}_{e}(\Ss_{2}), \quad i=1,2$$ with $ \pi_{1} \circ \e_{1} = \pi_{2} \circ \e_{2}$
so that $C^{*}_{e}(\Ss_{1}) *_{\Aa} C^{*}_{e}(\Ss_{2})$ is generated by $ \pi_{1}(C^{*}_{e}(\Ss_{1})) \cup \pi_{2}(C^{*}_{e}(\Ss_{2})) $ and the following universal property holds: any two unital *-homomorphisms $\rho_{i} : C^{*}_{e}(\Ss_{i}) \rightarrow \cl B(H), \ i =1, 2$ with $ \rho_{1} \circ \e_{1} = \rho_{2} \circ \e_{2}$, give rise to a *-representation $ \rho : C^{*}_{e}(\Ss_{1}) *_{\Aa} C^{*}_{e}(\Ss_{2}) \rightarrow \cl B(H)$ such that $ \rho \circ \pi_{i} = \rho_{i}$, $ i=1, 2$.

 Omit the maps $\pi_{i}$ and simply consider $ C^{*}_{e}(\Ss_{i})\subseteq  C^{*}_{e}(\Ss_{1}) *_{\Aa} C^{*}_{e}(\Ss_{2}) $, $ i =1,2$. Denote by  $\phi_i : \Ss_i \rightarrow C^*_e(\Ss_i), \ i=1,2$ the canonical completely isometric embeddings.
By Proposition \ref{intheircsalgebras} 
there exists a unital complete order isomorphism $\Phi$ from  $  \Ss_{1}\oplus_{\cl A}\Ss_{2}$ onto $  \phi_{1}(\Ss_{1}) + \phi_{2}(\Ss_{2}) \subseteq   C^{*}_{e}(\Ss_{1}) *_{\Aa} C^{*}_{e}(\Ss_{2})$ such that  $ \Phi|_{\Ss_i} = \phi_i $, $i =1,2$.
 It is clear that $( C^{*}_{e}(\Ss_{1}) *_{\Aa} C^{*}_{e}(\Ss_{2}),\Phi) $  is a $C^*$-cover for $ \Ss_{1}\oplus_{\Aa}\Ss_{2}$  and thus by Proposition \ref{allcstarcovers}, it suffices to show that the coproduct is hyperrigid in it.

 Note that since the $\Ss_i$ are hyperrigid, for any representations $ \sigma_i : C^{*}_{e}(\Ss_{i}) \to \cl B(H)$, $ \sigma_{i} \circ \phi_{i}$, $ i =1,2$, has the unique extension property; equivalently, $ \sigma_{i} \circ \phi_{i}$, $ i =1,2$ is a  maximal unital completely positive map (see \cite[Proposition 2.2]{Arveson2003NOTESOT}.

 To this end, let $\sigma : C^{*}_{e}(\Ss_{1}) *_{\Aa} C^{*}_{e}(\Ss_{2}) \to \cl B(H) $ be a representation and let $ \Psi:= \sigma|_{\Ss_{1}\oplus_{\Aa}\Ss_{2}}= \sigma \circ \Phi$. By the discussion above,  $ \sigma_{i}:= \sigma|_{C^{*}_{e}(\Ss_{i})} $, $ i=1,2$ are representations with the unique extension property. We will show that $ \Psi$ is maximal.
Indeed,  suppose that $\Phi$ has a unital completely positive dilation,   that is, a ucp map $\psi : \Ss_{1}\oplus_{\Aa} \Ss_{2} \rightarrow \B(K) $ for some Hilbert space $K$ containing $H$ such that $\Psi(a) = P_{H}\psi(a)\arrowvert_{H}$, for all $a \in \Ss_{1}\oplus_{\Aa} \Ss_{2}$, where $P_H\in\cl B(K)$ is the orthogonal projection onto $H$ (we write  $\psi \geq \Phi $). Now we see that if $a_{i} \in \Ss_{i}$, $i=1, 2$,
\[ \sigma_{i}|_{\Ss_{i}}(a_{i}) = \Psi \arrowvert_{\Ss_{i}}(a_{i}) = P_{H}\psi \arrowvert_{\Ss_{i}}(a_{i})\arrowvert_{H} \]
and thus $ \psi \arrowvert_{\Ss_{i}} \geq \sigma_{i}|_{\Ss_{i}}$. But since $\sigma_{i}|_{\Ss_{i}}$ are maximal, this is equivalent (see \cite[Proposition 2.2]{Arveson2003NOTESOT}) to 
\[  \sigma_{i}|_{\Ss_{i}}(a_{i})x= \psi \arrowvert_{\Ss_{i}}(a_{i})x \] 
for all $a_{i} \in \Ss_{i}$ and $x \in H$. 

It now follows  that, \begin{align*}
     \Psi(a_{1} + a_{2})x & = \sigma_{1}|_{\Ss_{1}}(a_{1})x + \sigma_{2}|_{\Ss_{2}}(a_{2})x  \\
     & =  \psi \arrowvert_{\Ss_{1}}(a_{1})x + \psi \arrowvert_{\Ss_{2}}(a_{2})x\\
     & = \psi(a_{1} + a_{2})x
\end{align*}    
for all $a_{1} \in \Ss_{1}$, $ a_{2}\in \Ss_{2}$ and $ x \in H$. Hence $\Psi(a_{1} + a_{2})= \psi(a_{1} + a_{2})$ and so $\Psi$ and $\psi$ coincide on $\Ss_{1}+\Ss_{2}=\Ss_{1}\oplus_{\Aa} \Ss_{2}$, showing that  $\Phi $ is maximal, equivalently, it has the unique extension property. We have thus showed that for an arbitrary representation $ \sigma$, the restriction  $ \sigma|_{\Ss_{1}\oplus_{\Aa} \Ss_{2}} $ has the unique extension property, that is, $ \Ss_{1}\oplus_{\Aa} \Ss_{2}$ is hyperrigid in  $C^{*}_{e}(\Ss_{1}) *_{\Aa} C^{*}_{e}(\Ss_{2})$. By Theorem \ref{allcstarcovers}, $C^{*}_{e}(\Ss_{1}  \oplus_{\Aa} \Ss_{2}) \cong C^{*}_{e}(\Ss_{1}) *_{\Aa} C^{*}_{e}(\Ss_{2})$.

\end{proof}

\section{A SPECIAL OPERATOR A-SYSTEM AND DUAL OPERATOR A-SYSTEMS} \label{graphs}

We are now able to go back to the examples of graph operator systems and answer the question on whether the coproduct of two graph operator systems is again a graph operator system (since it is a $\mathcal{D}_{n}$-bimodule).

\begin{proposition} \label{toantiparad}
The coproduct of two graph operator systems is not necessarily a graph operator system. That is, there exist two graph operator systems whose coproduct is not  completely order isomorphic to any operator system $  \Ss_{G'} \subseteq M_{k}$ that is a bimodule over $ \Dd_{k}$ for some $ k \in \N$.   In fact, this coproduct cannot even be completely order isomorphic to  any operator system acting on a finite dimensional Hilbert space.
\end{proposition}
\begin{proof}
Let $G$ be the complete graph on 2 vertices. Let $\Ss_{G}$ be its graph operator system and consider two copies of $ \Ss_{G}$.  Let $ \Ss_{G} \oplus_{\Dd_{2}} \Ss_{G}$  be their coproduct over $ \Dd_{2}$. Note that the graph operator system $\Ss_{G}$ is in fact all 2 by 2 scalar matrices, i.e. $ \Ss_{G}= M_{2}$ and $\Ss_{G} \oplus_{\Dd_{2}} \Ss_{G} = M_{2} \oplus_{\Dd_{2}} M_{2} $ .

Assume that $M_{2} \oplus_{\Dd_{2}} M_{2}  $ is completely order isomorphic with a graph operator system $ \Ss_{G'}$ with $ \Dd_{k} \subseteq \Ss_{G'} \subseteq M_{k}$ that is a bimodule over $ \Dd_{k}$ for some $ k \in \N$. Hence 
$$ C_{e}^{*}(M_{2} \oplus_{\Dd_{2}} M_{2} ) \cong C^{*}_{e}(\Ss_{G'}).$$

 Since $ M_{2}$ clearly contains enough unitaries to generate its $\Cs$-envelope (i.e. itself) it is hyperrigid by Proposition \ref{hyperrigidimpliesunitaries}. Hence, by Theorem \ref{C*envelopesofAsystems} we have that 
$C_{e}^{*}(M_{2} \oplus_{\Dd_{2}} M_{2} ) \cong M_{2} *_{\Dd_{2}} M_{2} $.

Finally, by \cite[Theorem 3.2]{ORTIZ2015128}, we have $ C^{*}_{e}(\Ss_{G'})=C^{*}(\Ss_{G'}) \subseteq M_{k}$ and thus we have following *-isomorphism
\begin{align*}
    M_{2} *_{\Dd_{2}} M_{2} \cong C^{*}(\Ss_{G'}).
\end{align*}
However, this  cannot hold since the amalgamated free product on the left hand side is infinite dimensional while the right hand side is finite dimensional. In particular $M_{2}  \oplus_{\Dd_{2}} M_{2}$ cannot even be represented concretely as an operator system in a finite dimensional Hilbert space.
\end{proof}

So, $M_{2}  \oplus_{\Dd_{2}} M_{2}$ is not a graph operator system despite the fact that it is a bimodule over the 2 by 2 diagonal matrices. However, as we will see shortly, it is a dual operator $\Dd_{2}$-system.

In fact, the coproduct of any two dual operator $\Aa$-systems is a dual operator $\Aa$-system. In order to see this, we need to prove first that their coproduct is w*-closed, and also that the module action is separately w*-continuous.
For more about the subject, we refer to the work of Y-F. Lin and I. Todorov \cite{10.1093/imrn/rnz364}.

\begin{definition}\cite{10.1112/blms/bdq103}
Let $\Ss$ be an operator system. We say that $\Ss$ is a dual operator system, if it is a dual operator space. That is, if there exists an operator space $ \Ss_{*}$ such that $ \Ss \cong (\Ss_{*})^{*}$ completely isometrically isomorphically.
\end{definition}

\begin{definition}\cite[Definition 4.3]{10.1093/imrn/rnz364}
Let $\Aa$ be a von Neumann algebra. An operator system $\Ss$, will be called a dual operator $\Aa$-system if
\begin{enumerate}
    \item $\Ss$ is an operator $\Aa$-system
    \item $\Ss$ is a dual operator system, and
    \item The map from $\Aa \times \Ss$ into $\Ss$, that sends the pair $ (a,s)$ to $ a \cdot s$, is separately w*-continuous.
\end{enumerate}
\end{definition}

\begin{theorem}\cite[Theorem 4.7]{10.1093/imrn/rnz364} \label{yingfenivan}
Let $\Aa$ be a von Neumann algebra and $\Ss$ be a dual operator $\Aa$-system. There exists a Hilbert space $ \Hh$, a unital complete order embedding $ \gamma : \Ss \rightarrow \Bh$  which is a w*-homeomorphism with w*-closed range, and a unital normal *-homomorphism $ \pi : \Aa \rightarrow \Bh$, such that, 
\begin{align*}
    \gamma(a \cdot s) = \pi(a) \gamma(s), \; \; \text{ for all } \; \; a \in \Aa, s \in \Ss.
\end{align*}
\end{theorem}

\begin{theorem} \label{coproductsofduals}
Let $\Aa \subseteq \cl B(H)$ be a von Neumann algebra and let $\Ss, \Tt \subset \cl B(H)$ be two w*-closed operator systems that are bimodules over $\Aa$ and such that $ \Aa $ in contained in both. Then, their coproduct $ \Ss \oplus_{\Aa} \Tt$ is also a dual operator $ \Aa$-system.
\end{theorem} 
\begin{proof}
By Theorem \ref{corollaryvN} we can  write the coproduct as $ \Ss \oplus_{\Aa} \Tt = \Ss \oplus \Tt/ \mathcal{J}$, where $ \mathcal{J} =\{a \oplus -a : a\in \Aa \} \subseteq \B(H \oplus H)$.   Now $ \Ss \oplus \Tt$ is a w*-closed subspace of $ \B(H \oplus H)$ and $ \mathcal{J}$ is a w*-closed subspace of $ \Ss \oplus \Tt$.

If $ Y$ is a dual operator space and $X \subseteq  Y$ a weak* continuous subspace, then $ Y/X = (X_{\perp})^{\rm d}$ completely isometrically. Thus, $\Ss \oplus \Tt/ \mathcal{J} \cong (\mathcal{J_{\perp}})^{*} $ as operator spaces. Here  $$\mathcal{J_{\perp}} = \{u \oplus v \in \Ss_{*} \oplus \Tt_{*}: \langle u\oplus v ,a \oplus -a \rangle 
=0 \text{ for all } a \in \Aa  \}.$$  So the coproduct is a dual operator system. It is by definition an operator $\Aa$-system, so in order for it to be a dual operator $\Aa$-system we need to check that the module action is separately w*-continuous.

We recall  that the completely isometric isomorphism $\Ss \oplus \Tt/ \mathcal{J} \cong (\mathcal{J_{\perp}})^{\rm d} $ is the map:
 \begin{align*}
     \Phi : \Ss \oplus \Tt/ \mathcal{J} \to  (\mathcal{J_{\perp}})^{\rm d} : \;
     x= s\oplus t + \mathcal{J}& \mapsto \Phi_x
     \end{align*}
     where $\Phi_x(u\oplus v)= \langle u\oplus v, s\oplus t\rangle, \; \; u \oplus v \in \mathcal{J_{\perp}}$.
 
 Suppose that $ (a_{i})_{i\in I}$ is a net in $\Aa$ such that $ a_{i} \xrightarrow{w^{*}} a \in \Aa $. Let also $ x = s\oplus t + \mathcal{J} \in \Ss \oplus \Tt/ \mathcal{J}$. We will show that $ a_{i} \cdot x \xrightarrow{w^{*}} a \cdot x$. Note that 
 \begin{align*}
 a_{i} s \oplus  a_{i}t & \xrightarrow{w^{*}} as \oplus at \text{ in } \Ss \oplus \Tt \Rightarrow \\
 \langle u  \oplus v ,a_{i} s \oplus a_{i}t  \rangle  & \rightarrow   \langle u  \oplus v ,a_{} s \oplus a_{}t  \rangle, \;  \forall  u \oplus v \in \mathcal{J_{\perp}} \Rightarrow \\ 
     \Phi_{a_{i} \cdot x }(u \oplus v) & \rightarrow \Phi_{a \cdot x}(u \oplus v), \;  \forall u \oplus v \in \mathcal{J_{\perp}}.
 \end{align*}
 So, it suffices to show that $a_{i} s \oplus  a_{i}t  \xrightarrow{w^{*}} as \oplus at$.
 
 This holds since $ a_{i} \xrightarrow{w^{*}} a $ in  $ \Aa $ implies that $ a_{i} s \xrightarrow{w^{*}} a s $ 
 and $ a_{i} t \xrightarrow{w^{*}} a t $ by the separate w*-continuity of multiplication in $ \cl B(H)$. 
 
 Now assume that $ x_{i} \xrightarrow{w^{*}} x$ in  $ \Ss \oplus \Tt / \mathcal{J}$ and let $ a\in \Aa$. We note that whenever $ u\oplus v \in \cl J_{\perp}$, where $ u \in \Ss_{*}$ and $ v \in \Tt_{*}$, then $ u \cdot a \in \Ss_{*}$ and $  v \cdot a\in \Tt_*$.  Here we denote by $S^{1}(H)$ the trace class operators on $H$. Indeed, from the identification $ \Ss_{*}= S^{1}(H)/\Ss_{\perp}$ (resp. $ \Tt_{*}= S^{1}(H)/\Tt_{\perp}$) write $ u = T_u + \Ss_{\perp}$ and note that the action $(T_{u} + \Ss_{\perp})\cdot a = T_{u}a + \Ss_{\perp}$ is well defined. This follows from the fact that $ \Aa \cdot \Ss \subseteq \Ss$. Similarly for $ \Tt_{*}$. Noting that if $ u\oplus v \in \cl J_{\perp}$ then $ (u\cdot a)\oplus (v \cdot a ) \in \cl J_{\perp}$.  Hence, for $ x_i = s_i \oplus t_i + \cl J$ and $ x = s\oplus t+\cl J$,
 \begin{align*}
     x_i & \xrightarrow{w*} x \text{ in } \Ss \oplus \Tt \Rightarrow  \\
     \langle u  \oplus v , s_i \oplus t_i  \rangle  & \rightarrow   \langle u  \oplus v , s \oplus t  \rangle, \;  \forall  u \oplus v \in \mathcal{J_{\perp}}. 
 \end{align*}
 Hence, for an arbitrary $ u\oplus v \in \cl J_{\perp}$,
 \begin{align*}
     \langle(u\cdot a)\oplus (v \cdot a ), s_i \oplus t_i  \rangle  & \rightarrow   \langle (u\cdot a)\oplus (v \cdot a ), s \oplus t  \rangle \Rightarrow \\
      \langle u  \oplus v, as_i \oplus at_i  \rangle  & \rightarrow   \langle u  \oplus v, as \oplus at  \rangle 
 \end{align*}
 showing that $ a\cdot x_i \to a\cdot x$.  Thus  $\Ss \oplus \Tt / \mathcal{J} = \Ss \oplus_{\Aa} \Tt$ is a dual operator $\Aa$-system. 
\end{proof}

\section{INDUCTIVE LIMITS} \label{inductivelimits}

We now move on to the topic of \textit{inductive limits}, to see how well the coproducts behave. For an extensive study of inductive limits for operator systems see \cite{Mawhinney2017InductiveLI}. First, we state the definition of an inductive system and inductive limit in the operator system category.
\begin{definition}
An inductive system in the category of operator systems, is a pair $(\{\Ss_{k}\}_{k\in \N}, \{\phi_{k}\}_{k \in \N})$ where  $ \Ss_{k}$ is an operator system for each $k \in \N$ and $ \phi_{k}:\Ss_{k} \rightarrow \Ss_{k+1}$ is a unital completely positive map. 

An inductive limit, for the inductive system $(\{\Ss_{k}\}_{k\in \N}, \{\phi_{k}\}_{k \in \N})$, is a pair $(\Ss, \{\phi_{k,\infty}\}_{k \in \N})$ where $ \Ss$ is an operator system and $\phi_{k, \infty} : \Ss_{k} \rightarrow \Ss $ is a u.c.p. map for each $ k \in \N$ such that:
\begin{enumerate}
    \item $\phi_{k+1, \infty} \circ \phi_{k} = \phi_{k, \infty}$, 
    \item if $(\Tt, \{ \psi_{k}\}_{k\in \N})$ is another pair such that $ \Tt$ is an operator system, $ \psi_{k}: \Ss_{k} \rightarrow \Tt$ is a u.c.p. map and $ \psi_{k+1} \circ \phi_{k} = \psi_{k}$, $ k \in \N$, then there exists a unique u.c.p. map $ \mu : \Ss \rightarrow \Tt$ such that $ \mu \circ \phi_{k, \infty} = \psi_{k}$, for each $k \in \N$.
\end{enumerate}
\end{definition}

We will show that any two inductive systems in the category of operator systems, give rise to an inductive system of  coproducts of the operator systems at each term, having as an inductive limit, the coproduct of the inductive limits of the two inductive systems.

\begin{proposition}
Let $(\Ss, \{\phi_{k,\infty}\}_{k \in \N})$ and $ (\Tt, \{\psi_{k,\infty}\}_{k \in \N})$ be two inductive limits in the operator system category, of the inductive systems:
\begin{align*}
    \Ss_{1} \xrightarrow{\phi_{1}} \Ss_{2} \xrightarrow{\phi_{2}} & \Ss_{3}  \xrightarrow{\phi_{3}} \Ss_{4}\xrightarrow{\phi_{4}} \cdots\\
    &\text{and}\\
    \Tt_{1} \xrightarrow{\psi_{1}} \Tt_{2} \xrightarrow{\psi_{2}}& \Tt_{3}  \xrightarrow{\psi_{3}} \Tt_{4}\xrightarrow{\psi_{4}} \cdots
\end{align*}
respectively. Then, the pair $(\Ss\oplus_{1}\Tt,\{\phi_{k,\infty}\oplus_{1}\psi_{k,\infty}\}_{k \in \N})$ is an inductive limit for the inductive system
\begin{align*}
    \Ss_{1}\oplus_{1}\Tt_{1} \xrightarrow{\phi_{1}\oplus_{1}\psi_{1}} \Ss_{2}\oplus_{1}\Tt_{2} \xrightarrow{\phi_{2}\oplus_{1}\psi_{2}} & \Ss_{3}\oplus_{1}\Tt_{3}  \xrightarrow{\phi_{3}\oplus_{1}\psi_{3}} \Ss_{4}\oplus_{1}\Tt_{4}\xrightarrow{\phi_{4}\oplus_{1}\psi_{4}} \cdots
\end{align*}
 where for each $k \in \N$, the map $ \phi_{k}\oplus_{1}\psi_{k}: \Ss_{k}\oplus_{1}\Tt_{k} \rightarrow \Ss_{k+1}\oplus_{1}\Tt_{k+1}$ is the unique u.c.p. map associated with the maps $ \phi_{k}$ and $ \psi_{k}$ after we embed both $ \Ss_{k+1}$ and  $ \Tt_{k+1}$ in $  \Ss_{k+1}\oplus_{1}\Tt_{k+1}$. 
 \begin{proof}
   Let us begin by considering each one of the operator systems $ \Ss_{k}$ and $ \Tt_{k}$ as  subsystems of some $\Cs$-algebras, say their universal $\Cs$-algebras $ C_{u}^{*}(\Ss_{k})$ and $ C_{u}^{*}(\Tt_{k})$ respectively, for every $ k \in \N$. Then, for every $k \in \N$, we can identify (completely order isomorphically) their coproducts $\Ss_{k}\oplus_{1}\Tt_{k}$ with the subsystems $ \Ss_{k} + \Tt_{k} \subseteq C_{u}^{*}(\Ss_{k}) *_{1} C_{u}^{*}(\Tt_{k})$ as in Proposition \ref{suminuniversals}, with the associated maps being the inclusions $ \Ss_{k}, \;  \Tt_{k} \hookrightarrow \Ss_{k} + \Tt_{k}$. We thus assume that each coproduct $ \Ss_{k} \oplus_{1} \Tt_{k}$ is spanned by sums of the form $ s_{k} + t_{k}$, for $ s_{k} \in \Ss_{k}$, $ t_{k} \in \Tt_{k}$. So, we have the following inductive systems: 
    \begin{center}
          \begin{tikzcd}[row sep=large, column sep=large]
             \Ss_{1} \arrow[d, hook]   \arrow[r, "\phi_{1}"]  & \Ss_{2} \arrow[d, hook]   \arrow[r, "\phi_{2}"]  & \Ss_{3} \arrow[d, hook]  \arrow[r, "\phi_{3}"]  & \cdots\\
             \Ss_{1}\oplus_{1}\Tt_{1} \arrow[r, "\phi_{1}\oplus_{1}\psi_{1}"]  & \Ss_{2}\oplus_{1}\Tt_{2} \arrow[r, "\phi_{2}\oplus_{1}\psi_{2}"] &\Ss_{3}\oplus_{1}\Tt_{3} \arrow[r, "\phi_{3}\oplus_{1}\psi_{3}"] & \cdots\\
              \Tt_{1} \arrow[u, hook]   \arrow[r, "\psi_{1}"]  & \Tt_{2} \arrow[u, hook]   \arrow[r, "\psi_{2}"]  & \Tt_{3} \arrow[u, hook]  \arrow[r, "\psi_{3}"]  & \cdots
              \end{tikzcd}
                \end{center}
 where for each $k \in \N$, $\phi_{k}\oplus_{1}\psi_{k}: \Ss_{k}\oplus_{1} \Tt_{k} \rightarrow \Ss_{k+1}\oplus_{1} \Tt_{k+1}$ is the unique u.c.p. map, that satisfies
\begin{align} \label{equatrestrict}
    (\phi_{k}\oplus_{1}\psi_{k})\arrowvert_{\Ss_{k}} = \phi_{k} \; \text{ and } \; 
    (\phi_{k}\oplus_{1}\psi_{k})\arrowvert_{\Tt_{k}} = \psi_{k}.
\end{align}
 Moreover, if we do the same for the operator systems $\Ss$ and $\Tt$, we also obtain u.c.p. maps
$$ \phi_{k,\infty}\oplus_{1} \psi_{k,\infty} : \Ss_{k}\oplus_{1}\Tt_{k} \rightarrow \Ss \oplus_{1}\Tt, $$ such that 
\begin{align} \label{deytero}
    (\phi_{k,\infty}\oplus_{1}\psi_{k,\infty})\arrowvert_{\Ss_{k}} = \phi_{k,\infty} \; \text{ and } \; 
    (\phi_{k,\infty}\oplus_{1}\psi_{k,\infty})\arrowvert_{\Tt_{k}} = \psi_{k,\infty},
\end{align}
for each $ k \in \N$.
In order to prove that $(\Ss\oplus_{1}\Tt,\{\phi_{k,\infty}\oplus_{1}\psi_{k,\infty}\}_{k \in \N})$ is indeed an inductive limit in the operator system category, we have to verify the following conditions:

1)  $ (\phi_{k+1,\infty}\oplus_{1}\psi_{k+1,\infty} )\circ (\phi_{k}\oplus_{1}\psi_{k}) = \phi_{k,\infty}\oplus_{1}\psi_{k,\infty} $,

2) [\textit{universal property}] If $(\Rr , \{\rho_{k}\}_{k \in \N})$ is another pair, such that $ \Rr$ is an operator system, $ \rho_{k} : \Ss_{k}\oplus_{1}\Tt_{k}  \rightarrow \Rr$ is a u.c.p. map with $ \rho_{k+1} \circ (\phi_{k}\oplus_{1}\psi_{k}) = \rho_{k}$, $ k \in \N$, then there exists a unique u.c.p. map $ \nu : \Ss \oplus_{1} \Tt \rightarrow \Rr$ such that $ \nu \circ (\phi_{k,\infty}\oplus_{1}\psi_{k,\infty}) = \rho_{k}$.

For  property 1), fix $k \in \N$ and let $ s_{k} \in \Ss_{k}$, $ t_{k} \in \Tt_{k}$. Thus, 
\begin{align*} 
     & \Big( (\phi_{k+1,\infty}\oplus_{1}\psi_{k+1,\infty} )\circ (\phi_{k}\oplus_{1}\psi_{k}) \Big) (s_{k} + t_{k})  =  \\
     &= (\phi_{k+1,\infty}\oplus_{1}\psi_{k+1,\infty} )\circ \Big( (\phi_{k}\oplus_{1}\psi_{k}) (s_{k}) +(\phi_{k}\oplus_{1}\psi_{k}) (t_{k})\Big) \\
   & = (\phi_{k+1,\infty}\oplus_{1}\psi_{k+1,\infty} )\circ \Big( \phi_{k} (s_{k}) +\psi_{k} (t_{k})\Big)\\
    & = (\phi_{k+1,\infty}\oplus_{1}\psi_{k+1,\infty} )\circ \phi_{k} (s_{k})+ (\phi_{k+1,\infty}\oplus_{1}\psi_{k+1,\infty} )\circ  \psi_{k} (t_{k})\\
     &= (\phi_{k+1,\infty} \circ \phi_{k}) (s_{k})+ (\psi_{k+1,\infty} \circ  \psi_{k}) (t_{k})\\
     &= \phi_{k,\infty}(s_{k}) + \psi_{k,\infty}(t_{k})\\
     &= \phi_{k,\infty}\oplus_{1}\psi_{k,\infty}(s_{k} + t_{k}),
\end{align*} 
where the third line is due to \ref{equatrestrict}, the fifth line comes from \ref{deytero}  and the sixth is from the  corresponding property of the inductive limits $(\Ss, \{\phi_{k,\infty}\}_{k \in \N})$ and $(\Tt, \{\psi_{k,\infty}\}_{k \in \N})$. So property 1) holds since, as mentioned, sums of the form $ s_{k}+t_{k}$ span the coproduct $\Ss_{k}\oplus_{1}\Tt_{k}$ for every $ k \in \N$. 

It remains to prove the universal property 2), and then by uniqueness, the pair $(\Ss\oplus_{1}\Tt,\{\phi_{k,\infty}\oplus_{1}\psi_{k,\infty}\}_{k \in \N})$  will  be the inductive limit. To this end, let $ (\Rr,\{\rho_{k}\}_{k \in \N})$ be a pair as in 2). Consider, for each $ k \in \N$ 
\[ \rho_{k+1}: \Ss_{k+1}\oplus_{1}\Tt_{k+1} \rightarrow \Rr \] 
then 
\begin{align*}
    \rho_{k+1}\arrowvert_{\Ss_{k+1}}: \Ss_{k+1}\rightarrow \Rr\\
    \rho_{k+1}\arrowvert_{\Tt_{k+1}}: \Tt_{k+1}\rightarrow \Rr,
\end{align*}
 are two u.c.p. maps such that 
 \begin{align*}
     \rho_{k+1}\arrowvert_{\Ss_{k+1}} \circ \phi_{k} = \rho_{k}\arrowvert_{\Ss_{k}}\; \text{ and } \; \rho_{k+1}\arrowvert_{\Tt_{k+1}} \circ \psi_{k} = \rho_{k}\arrowvert_{\Tt_{k}}.
 \end{align*}
 Therefore, by the universal properties of the inductive limits $(\Ss, \{\phi_{k,\infty}\}_{k \in \N})$ and $ (\Tt, \{\psi_{k,\infty}\}_{k \in \N})$, there exist two u.c.p. maps 
 \begin{align*}
     \mu: \Ss \rightarrow \Rr \\
     \lambda :\Tt \rightarrow \Rr,
 \end{align*}
 such that for every $k \in \N$
 \begin{align*}
     \mu \circ \phi_{k,\infty} = \rho_{k}\arrowvert_{\Ss_{k}}\\
     \lambda \circ \psi_{k,\infty} = \rho_{k}\arrowvert_{\Tt_{k}}.
 \end{align*}
 Finally, if we invoke the universal property of the coproducts (Definition \ref{defAcoprod}), we obtain a unique u.c.p. map 
 \begin{align*}
     \nu: \Ss\oplus_{1}\Tt \rightarrow \Rr,
 \end{align*}
 that satisfies $ \nu\arrowvert_{\Ss}= \mu$ and $ \nu\arrowvert_{\Tt}= \lambda$ and so, for every $k\in \N$, $ s_{k} \in \Ss_{k}$ and $ t_{k} \in \Tt_{k}$
 \begin{align*}
  \Big(\nu \circ (\phi_{k,\infty} \oplus_{1}\psi_{k,\infty})\Big)(s_{k}+t_{k}) &=  \nu \circ \Big( \phi_{k,\infty}(s_{k}) + \psi_{k,\infty}(t_{k}) \Big)\\
  &= \nu \circ \phi_{k,\infty}(s_{k})+  \nu \circ \psi_{k,\infty}(t_{k})\\
  & = \mu \circ \phi_{k,\infty}(s_{k})+  \lambda \circ \psi_{k,\infty}(t_{k})\\
  & = \rho_{k}\arrowvert_{\Ss_{k}}(s_{k}) + \rho_{k}\arrowvert_{\Tt_{k}}(t_{k})\\
  &= \rho_{k}(s_{k} +t_{k})
 \end{align*}
 by the aforementioned properties.  Thus $ \nu \circ (\phi_{k,\infty}\oplus_{1}\psi_{k,\infty}) = \rho_{k}$ as we wanted.
 \end{proof}
\end{proposition}


\printindex
\clearpage

\addcontentsline{toc}{Section}{REFERENCES}
\bibliography{Bibliography}
\nocite{*} 
\bibliographystyle{acm}

{\small DEPARTMENT OF MATHEMATICS, NATIONAL AND KAPODISTRIAN UNIVERSITY OF ATHENS, ATHENS 157 84, GREECE}

\textit{E-mail address:} \texttt{achatzinik@math.uoa.gr}

\end{document}